\documentclass[12pt]{article}
\usepackage[a4paper,mag=1000,includefoot,left=2cm,right=2cm,top=2cm,bottom=2cm,headsep=1cm,footskip=1cm]{geometry}
\usepackage[T2A]{fontenc}
\usepackage[cp1251]{inputenc}
\usepackage[english]{babel}
\usepackage{caption}
\usepackage{pgfplotstable,booktabs,colortbl}
\usepackage{fixltx2e}           % workaround for usage of eqref in caption
\usepackage{amssymb,amsmath,amsthm,indentfirst,bm}
\usepackage{graphics,epsfig}
\ifpdf\usepackage{epstopdf}\fi
\usepackage{esint}          % Use nice integral symbols
\usepackage{mathtools}      % Add usefull package
\mathtoolsset{showonlyrefs} % Numbering only referenced equations
 \newtheorem{proposition}{Proposition}   % \newtheorem{proposition}{\indent Утверждение} %
 \newtheorem{corollary}{Corollary}       % \newtheorem{corollary}{\indent Следствие}     %
 \newtheorem{remark}{Remark}             % \newtheorem{remark}{\indent Замечание}        %
\makeatletter%              % Set small font size in the tables
 \renewenvironment{table}{%
    \@float{table}}
    {\end@float}
\makeatother
       \def\vk{\varkappa}

 \def\ov{\overline}   \def\wh{\widehat}   \def\wt{\widetilde}
 \def\ll{\left}       \def\rr{\right}     
 \def\leq{\leqslant}  \def\geq{\geqslant} \def\pa{\partial}
       
  \def\bPsi{\boldsymbol\Psi}     \def\bPhi{\boldsymbol\Phi}
 \DeclareMathOperator{\Rea}{Re}
 \DeclareMathOperator{\Ima}{Im}
 \newcommand{\abs}[1]{\ll\lvert#1\rr\rvert}
 
 \renewcommand{\leq}{\leqslant}
 \renewcommand{\geq}{\geqslant}
 \renewcommand{\raggedright}{\leftskip=0pt \rightskip=0pt plus 0cm}
\noeqref{eq:bc,eq:ic,eq:f102}% \noeqref{*}
\begin{document}
\begin{center}
{\large{\bf SOME REMARKS ON DISCRETE AND SEMI-DISCRETE}}
{\large{\bf TRANSPARENT BOUNDARY CONDITIONS FOR SOLVING}}
{\large{\bf THE TIME-DEPENDENT SCHR\"ODINGER EQUATION ON THE HALF-AXIS}}
% Some remarks on discrete and semi-discrete transparent boundary conditions for solving the time-dependent Schr\"odinger equation on the half-axis
\vskip 0.3cm
\textsc{Alexander Zlotnik}
\footnote{Department of Higher Mathematics at Faculty of Economics,
National Research University Higher School of Economics,
Myasnitskaya 20, 101000 Moscow, Russia.
%E-mail: \it{azlotnik2008@gmail.com}
}
{and \textsc{Ilya Zlotnik}}
\footnote{Settlement Depository Company, 2-oi Verkhnii Mikhailovskii proezd 9, building 2, 115419 Moscow, Russia.
%E-mail: \it{ilya.zlotnik@gmail.com}
}
\end{center}
\begin{abstract}
We consider the generalized time-dependent Schr\"odinger equation on the half-axis and a broad family of finite-difference schemes with the discrete transparent boundary conditions (TBCs) to solve it.
We first rewrite the discrete TBCs in a simplified form explicit in space step $h$.
Next, for a selected scheme of the family, we discover that the discrete convolution in time in the discrete TBC does not depend on $h$ and, moreover, it coincides with the corresponding convolution in the semi-discrete TBC rewritten similarly. This allows us to prove the bound for the difference between the kernels of the discrete convolutions in the discrete and semi-discrete TBCs (for the first time). Numerical experiments on replacing the discrete TBC convolutions by the semi-discrete one exhibit truly small absolute errors though not relative ones in general. The suitable discretization in space of the semi-discrete TBC for the higher-order Numerov scheme is also discussed.
\end{abstract}
\section{Introduction}%
The time-dependent Schr\"odinger equation is crucial in quantum mechanics and electronics, atomic and nuclear physics, wave physics, etc. It should be often solved in unbounded space domains.
Several approaches were developed and investigated for solving problems of such kind in 1D, see review
\cite{AABES08}.
Among them one exploits the so-called discrete (both in space and time) transparent boundary conditions (DTBCs) at artificial boundaries, see  \cite{A98,EA01,AES03,M04,SA08} and \cite{DZ06,DZ07,DZZ09,ZZ12}.
Their advantages are the complete absence of spurious reflections in practice as well as the rigorous mathematical background and relevant stability results in theory.
Earlier the semi-discrete (continuous in space and discrete in time) TBCs were also constructed and studied \cite{SD95,SY97,YFS01,AB03}; they are simpler in constructing and thus have wider range of applications.
\par In this paper, we consider the generalized time-dependent Schr\"odinger equation on the half-axis and a broad family of finite-difference schemes with the DTBCs to solve it studied previously in \cite{DZZ09}. The schemes are two-level symmetric (of the Crank-Nicolson type) in time and with a parametric average in space that allows to include into consideration a number of particular schemes of various origin.
We first rewrite the DTBCs in a simplified form explicit in space step $h$.
Next, for a selected scheme in the family, we discover that the discrete convolution in time in the DTBC does not depend on $h$ and, moreover, it coincides with the corresponding one in the semi-discrete TBC (SDTBC) rewritten preliminarily in the similar form. The latter unexpected fact allows us to prove the bound for the difference between the kernels of the discrete convolutions representing discrete and semi-discrete TBCs (what is done for the first time).
\par The results of numerical experiments on replacing the DTBCs by the semi-discrete one are also presented. In general, they exhibit that the corresponding absolute errors are truly small, uniformly in time and both in $L^2$ and $C$ space mesh norms, though this can be not the case for the relative ones. We also discuss the suitable discretization in space of the semi-discrete TBC for the higher-order Numerov scheme.

\section{Theoretical results}
\label{FFDS}

We consider the initial-boundary value problem for a generalized 1D time-dependent Schr\"o\-din\-ger equation on the half-axis
\begin{gather}
 i\hbar\rho D_t\psi=-\frac{\hbar^{\,2}}{2}D(BD\psi)+V\psi
 \ \ \mbox{for}\ \ x>0\ \ \mbox{and}\ \ t>0,
\label{eq:se}\\
%[1mm]
 \psi|_{x=0}=0\ \ \mbox{and}\ \ \int_0^{\infty} \left|\psi(x,t)\right|^2dx<\infty
 \ \ \mbox{for}\ t>0,
\label{eq:bc}\\
%[1mm]
 \left. \psi\right|_{t=0}=\psi^0(x)\ \ \mbox{for}\ \ x>0.
\label{eq:ic}
\end{gather}
Hereafter the unknown wave function $\psi=\psi(x,t)$ is complex-valued,
$i$ is the imaginary unit, $\hbar>0$ is a physical constant and $\rho(x)$, $B(x)$ and
$V(x)$ are the given real-valued coefficients such that $\rho\geq \underline{\rho}>0$ and
$B\geq \underline{B}>0$.
Also $D_t=\frac{\partial}{\partial t}$ and $D=\frac{\partial}{\partial x}$ are the partial derivatives.
\par We also assume that, for some (sufficiently large) $X_0>0$,
 \begin{equation}
 \rho(x)=\rho_{\infty}>0, \ B(x)=B_{\infty}>0, \ V(x)=V_{\infty}
 \ \ \mbox{and}\ \ \psi^0(x)=0\ \ \mbox{for}\ \ x\geq X_0,
 \label{eq:s21}
 \end{equation}
so that \eqref{eq:se} becomes the much simpler Schr\"odinger equation with constant coefficients
\begin{equation}
 i\hbar\rho_{\infty} D_t\psi
 =-\frac{\hbar^{\,2}}{2} B_{\infty}D^2\psi+V_{\infty}\psi
 \ \
 \mbox{for}\ \ x>X_0\ \ \mbox{and}\ \ t>0.
\label{eq:sec}
\end{equation}
\par We fix some $X>X_0$ and define a non-uniform mesh $\overline{\omega}_{h,\infty}$ in $x$ on $[0,\infty)$ with the
nodes $0=x_0<\dots <x_J=X<\dots$ and the steps $h_j:=x_j-x_{j-1}$ supposing that $h_J\leq X-X_0$ and $h_j=h\equiv h_J$ for $j\geq J$. Let $\omega_{h,\infty}:=\overline{\omega}_{h,\infty}\setminus\,\{0\}$ and $h_{j+1/2}:=\frac{h_j+h_{j+1}}{2}$.
We exploit the backward and modified forward difference quotients as well as the backward and forward averages in $x$
\[
 \overline{\partial}_xW_j:= \frac{W_j-W_{j-1}}{h_j},\ \
 \widehat{\partial}_xW_j:= \frac{W_{j+1}-W_{j}}{h_{j+1/2}},
\]
\[
 \overline{s}_xW_j:= \frac{W_{j-1}+ W_j}{2},\ \
 \hat{s}_x W_j:=\frac{h_j}{2h_{j+1/2}}\,W_{j}
 +\frac{h_{j+1}}{2h_{j+1/2}}\,W_{j+1}.
\]
We also recall the three-point averaged operator of multiplication by a real mesh function $\varkappa$
\[
 C_{\theta}[\varkappa]W_j:=\theta\,\frac{h_j}{h_{j+1/2}}\,\varkappa_jW_{j-1}
 +(1-2\theta)(\hat{s}_x \varkappa_j)W_j
 +\theta\,\frac{h_{j+1}}{h_{j+1/2}}\,\varkappa_{j+1}W_{j+1}
\]
depending on the real parameter $\theta$ \cite{DZZ09}.
\par We also define the uniform in time mesh $\overline{\omega}^{\,\tau}$ with the nodes $t_m=m\tau$, $m\geq 0$, and the step $\tau>0$; let $\omega^\tau:=\overline{\omega}^{\,\tau}\setminus\,\{0\}$.
We exploit the backward difference quotient, the symmetric average and the backward shift in time
\[
 \overline{\partial}_tY:= \frac{Y-\check{Y}}{\tau},\ \
 \overline{s}_tY:= \frac{\check{Y}+Y}{2},\ \
 \check{Y}^m:=Y^{m-1}.
\]
\par In \cite{DZZ09}, a broad family of two-level symmetric in time (i.e., of the Crank-Nicolson type) finite-difference schemes was studied for problem \eqref{eq:se}-\eqref{eq:s21}
\begin{gather}
 i \hbar C_{\theta}[\rho_h]\overline{\partial}_t \Psi
 =-\frac{\hbar^2}{2}\,\hat{\partial}_x \left(B_h \overline{\partial}_x\overline{s}_t\Psi\right)+C_{\theta}[V_h]W\overline{s}_t\Psi\ \ \mbox{on}\ \ \omega_{h,\infty}\times \omega^{\tau},
\label{s23}\\
\Psi|_{j=0}=0\ \ \mbox{on}\ \  \omega^{\tau},\ \
\Psi^0=\Psi^0_h\ \ \mbox{on}\ \ \overline{\omega}_{h,\infty}.
\label{s210}
\end{gather}
Here $\rho_h,B_h$ and $V_h$ are (real) approximations of $\rho,B$ and $V$; we suppose that $\rho_h\geq \underline{\rho}$ and
$B_h\geq \underline{B}$. In the simplest case, one can set $\varkappa_{h_j}:=\varkappa(x_{j-1/2})$ for continuous $\varkappa=\rho,B$ and $V$.
\par For different values of $\theta$, the family includes a number of particular schemes:
the standard Crank-Nicolson scheme without averages (for $\theta=0$) studied in \cite{A98,EA01,DZ06,DZ07},
the finite element method (FEM) for linear elements (for $\theta=\frac{1}{6}$) studied in particular in \cite{AB03,SY97},
a four-point symmetric vector (or multi-symplectic) scheme (for $\theta=\frac{1}{4}$) studied in equivalent forms in \cite{HJW05,HLMZ05}
and, in the case of constant coefficients (for $\theta=\frac{1}{12}$), the higher-order Numerov scheme presented in \cite{M04,ZL10} (see also the 2D case in \cite{SA08}).
The case $\theta=\frac{1}{4}$ corresponds also to the linear FEM with the numerical integration based on the midpoint rule (in the integrals containing $\rho$ and $V$).
\par The uniform in time stability in two space norms was proved in \cite{DZZ09} for $\theta\leq\frac14$ that we suppose to be valid below.
\par The DTBC allows to restrict rigorously the decaying solution of a scheme on the infinite mesh to the finite in space mesh $\{x_j\}_{j=0}^J\times\overline{\omega}^{\,\tau}$. For scheme \eqref{s23}, \eqref{s210}, the DTBC was derived in \cite{DZZ09} in the form
\begin{gather}
 \frac{\hbar^2}{2}B_{\infty} \overline{\partial}_x \overline{s}_t \Psi_J^m
 -hs_{x\theta}^-\left(i \hbar \rho_{\infty}\overline{\partial}_t \Psi-V_{\infty} \overline{s}_t \Psi\right)_{J}^m
  =\frac{\hbar^2}{2}B_{\infty} {\mathcal S}_{\rm ref\,\theta}^m{\mathbf \Psi}_J^m\ \ \text{on}\ \ \omega^\tau,
\label{eq:p31}
\end{gather}
where  $s_{x\theta}^-W_J:=\theta W_{J-1}+\left(\frac{1}{2}-\theta\right)W_J$ and ${\mathbf \Psi}_J^m:=\left\{ \Psi_J^l\right\}_{l=1}^m$.
\begin{remark}
Notice that the left-hand side of \eqref{eq:p31} is the approximation to $\frac{\hbar^2}{2}B_{\infty}D\psi(X,(m-\frac{\tau}{2}))$ of the order $O(h^2+\tau^2)$ for any $\theta$ and even of the order $O(h^3+\tau^2)$ for $\theta=\frac16$.
\label{rem1}
\end{remark}
\par We intend to rewrite the operator ${\mathcal S}_{\rm ref\,\theta}$ in a simplified form explicit in $h$. Let $P_m(\mu)$ be the classical Legendre polynomials extended by $P_m(\mu)=0$ for $m<0$. We need the constants
\[
 \wh{a}=\wh{a}_0+i\,\wh{a}_1,\ \
 \wh{a}_0=\frac{V_{\infty}}{\hbar^2B_{\infty}},\ \
 \wh{a}_1=\frac{2\rho_{\infty}}{\tau\hbar B_{\infty}}>0
\]
independent of $h$.
Let $\arg z$ be defined up to $2\pi k$ for any integer $k$
whereas $\arg_0z\in[0,2\pi)$, for $z\in\mathbb{C}\setminus\{0\}$.
\begin{proposition}\label{prop:p5}
The operator in the DTBC \eqref{eq:p31} has the discrete convolution form
\begin{equation}\label{eq:p53}
 \mathcal{S}_{\rm ref\,\theta}^m\bPhi^m=c_{0\theta}(R(\vk_{\theta},\mu_{\theta})*\Phi)^m
 \equiv c_{0\theta}\sum\nolimits_{l=0}^m R^l(\vk_{\theta},\mu_{\theta})*\Phi^{m-l}
\end{equation}
for any $\Phi$: $\overline{\omega}^{\,\tau}\to {\mathbb C}$ such that $\Phi^0=0$, where
\begin{equation}\label{eq:p54}
 R^m(\vk,\mu):=-\frac{\vk^m}{2m-1}\ll[P_{m}(\mu)-P_{m-2}(\mu)\rr]\ \ \text{for}\ m\geq 0,
\end{equation}
with the parameters
\begin{gather}\label{eq:p57}
 c_{0\theta}=-\frac{|\wh{\alpha}_{\theta}|^{1/2}}{2} e^{-i\,(\arg_0 \wh{\alpha}_{\theta})/2},\ \
 \vk_{\theta}=-e^{i\arg\wh{\alpha}_{\theta}},\ \
 \mu_{\theta}=\frac{\wh{\beta}_{\theta}}{|\wh{\alpha}_{\theta}|}\in (-1,1),
\\[1mm] \label{eq:p58}
 \wh{\alpha}_{\theta}=2\wh{a}+(1-4\theta)h^2\,\wh{a}^2\neq 0,\ \
 \wh{\beta}_{\theta}=2\wh{a}_0+(1-4\theta)h^2|\wh{a}|^2.
\end{gather}
\end{proposition}
\begin{proof}
The presented formulas follow from respective ones in \cite{DZZ09} (refined from misprints) by inserting there $a=h^2\,\wh{a}$ excepting formula \eqref{eq:p57} for $c_{0\theta}$, where the sign minus should be replaced by $(-1)^{k_0}$ with the integer $k_0$ such that
\[
 \Delta_{\theta}:=2\arg_0(1-2\theta h^2\,\wh{a})-\arg_0\wh{\alpha}_{\theta}\in(2k_0\pi, 2(k_0+1)\pi)
\]
(the left $\arg_0$ could be replaced by $\arg$).
\par Let us prove that always $(-1)^{k_0}=-1$ for $\theta\leq\frac{1}{4}$ (note that for $\theta>\frac{1}{4}$ this is not the case in general). For $\theta<0$, we rewrite
\[
 \Delta_{\theta}=2\arg_0\zeta_{\theta}-
 \arg_0\left[\left(\zeta_{\theta}-\frac{1}{2\abs{\theta}}\right)\left(\zeta_{\theta}-\frac{1}{2\abs{\theta}(1+4\abs{\theta})}\right)\right],
\]
with $\zeta_{\theta}:=h^2\,\wh{a}+\frac{1}{2\abs{\theta}}$. Since $\arg_0\zeta_{\theta}\in(0,\pi)$, we get $\arg_0(\zeta_{\theta}-\delta)\in(\arg\zeta_{\theta},\pi)$ for $\delta>0$ and then $\Delta_{\theta}\in(-2\pi,0)$. Obviously $\Delta_{\theta}\in(-2\pi,0)$ for $\theta=0$, too.
\par For $0<\theta<\frac{1}{4}$, we write down
\[
 \Delta_{\theta}=2\arg_0\zeta_{\theta}-
 \arg_0\left[\left(\zeta_{\theta}-\frac{1}{2\theta}\right)\left(\zeta_{\theta}-\frac{1}{2\theta(1-4\theta)}\right)\right]
\]
with $\zeta_{\theta}:=\frac{1}{2\theta}-h^2\,\wh{a}$. Since now $\arg_0\zeta_{\theta}\in(\pi,2\pi)$, we get $\arg_0(\zeta_{\theta}-\delta)\in(\pi,\arg\zeta_{\theta})$ for $\delta>0$ and thus $\Delta_{\theta}\in(2\pi,4\pi)$.
\par Finally, for $\theta=\frac{1}{4}$, we get
$
 \Delta_{\theta}=2\arg_0\zeta_{\theta}-\arg_0\left(2-\zeta_{\theta}\right)
$
with $\zeta_{\theta}:=2-h^2\,\wh{a}$. We have $\arg_0\zeta_{\theta}\in(\pi,2\pi)$ and $\arg_0\left(2-\zeta_{\theta}\right)\in(0, \arg\zeta_{\theta}-\pi)$, thus $\Delta_{\theta}\in(2\pi,4\pi)$ too.
\end{proof}
\par Note that the fixed sign in formula \eqref{eq:p57} for $c_{0\theta}$ is essential to study asymptotic behavior as $h\to 0$ below.
\par We also rewrite \eqref{eq:p57} and \eqref{eq:p58} in a form similar to \cite{ZZ12}. Let $\wt{\alpha}_{\theta}:=2+(1-4\theta)h^2\,\wh{a}$, then
\begin{equation}\label{eq:p7a1}
 \wh{\alpha}_{\theta}=\wh{a}\,\wt{\alpha}_{\theta},\ \
 \arg_0\wh{\alpha}_{\theta} = \arg_0\wh{a}+\arg_0\wt{\alpha}_{\theta}\in(0,2\pi)
\end{equation}
since $0\leq \arg_0\wt{\alpha}_{\theta}<\arg_0\wh{a}<\pi$.
\begin{corollary}\label{cor:p7a}
Formulas \eqref{eq:p57} can be rewritten as
\begin{equation}\label{eq:p7a3}
\vk_{\theta}=-\exp\left\{i(\arg\wh{a}+\arg\wt{\alpha}_{\theta})\right\},\ \
\mu_{\theta}=\cos\left(\arg\wh{a}-\arg\wt{\alpha}_{\theta}\right).
\end{equation}
\end{corollary}
\begin{proof}
It suffices to note that
\[
 \frac{\wh{\beta}_{\theta}}{\abs{\wh{\alpha}_{\theta}}}
=\frac{\Rea(\wh{a}\,\wt{\alpha}_{\theta}^*)}{\abs{\wh{a}\,\wt{\alpha}_{\theta}}}
=\Rea e^{i(\arg\wh{a}-\arg\wt{\alpha}_{\theta})}.
\]
\par In addition, to make the derivation in \cite{DZZ09} closer to the FEM case \cite{ZZ12}, notice that, for the involved linear-fractional function
\[
 \gamma_{\theta}(z)=1+\frac{az+a^*}{bz+b^*}\ \ \text{with}\ \ b=1-2\theta a,
\]
one can write down
\[
 \gamma_{\theta}^2(z)-1=(\gamma_{\theta}(z)-1)(\gamma_{\theta}(z)+1)
=(\gamma_{\theta}^2(0)-1)\frac{\left(\frac{a}{a^*}z+1\right)\left(\frac{a+2b}{a^*+2b^*}z+1\right)}{\left(\frac{b}{b^*}z+1\right)^2}.
\]
The numerator of the fraction can be rewritten as
\[
 \left(\frac{a}{a^*}z+1\right)\left(\frac{a+2b}{a^*+2b^*}z+1\right)
=\left(e^{i\,2\arg\wh{a}}z+1\right)\left(e^{i\,2\arg\wt{\alpha}_{\theta}}z+1\right)
=(\vk_{\theta}z)^2-2\mu_{\theta}\vk_{\theta}z+1
\]
according to \cite{DZZ09}, and since
\[
 \left(e^{i\,2\varphi_1}z+1\right)\left(e^{i\,2\varphi_2}z+1\right)
=(-e^{i(\varphi_1+\varphi_2)}z)^2-2\cos(\varphi_1-\varphi_2)(-e^{i(\varphi_1+\varphi_2)})z+1
\]
for any real $\varphi_1$ and $\varphi_2$, formulas \eqref{eq:p7a3} appear once again.
\end{proof}
\par The next result is a direct consequence (see \cite{EA01,DZZ09}) of the classical Laplace asymptotic formula for the Legendre polynomials and the last corollary.
\begin{corollary}\label{cor:p7c}
The following asymptotic formula holds
\begin{gather*}
 R^m(\vk_\theta,\mu_\theta)=\frac{(-1)^m}{m^{3/2}}\sqrt{\frac{2}{\pi}\sin(\arg\wh{a}-\arg\wt{\alpha}_\theta)}\,e^{im(\arg\wh{a}+\arg\wt{\alpha}_\theta)}
\\[1mm]
 \times\cos\left(\left(m-\frac12\right)(\arg_0\wh{a}-\arg_0\wt{\alpha}_\theta)+\frac{3\pi}{4}\right)
 +O_\delta\left(\frac{1}{m^{5/2}}\right)
\end{gather*}
as $m\to\infty$ provided that $|\mu_\theta|\leq 1-\delta$ with some $\delta>0$.
\end{corollary}
\par This corollary is important to guarantee stable computations using $R^m$.
The condition imposed on $\mu_\theta$ in it can be specified as follows.
\begin{corollary}\label{cor:p7d}
Let $A\geq 1$ be a parameter. The following conditions
\begin{gather}\label{eq:p7d1}
 \frac{|\wh{a}_0|}{\wh{a}_1}=\frac{\tau |V_\infty|}{\hbar\rho_\infty}\leq A,\ \
 (1-4\theta)\frac{\rho_\infty}{\hbar B_\infty}\frac{h^2}{\tau}\leq A
\end{gather}
are sufficient for validity of $|\mu_\theta|\leq 1-\delta(A)$ with some $\delta(A)>0$.
The right condition is also necessary.
\end{corollary}
\begin{proof}
We have
\[
 \sin(\arg\wh{a}-\arg\wt{\alpha}_\theta)=\frac{\Ima(\wh{a}\wt{\alpha}_\theta^*)}{|\wh{a}\wt{\alpha}_\theta^*|}
 =\frac{2}{\frac{|\wh{a}|}{\wh{a}_1}|\wt{\alpha}_\theta|}.
\]
Furthermore
\begin{gather*}
\frac{|\wh{a}|}{\wh{a}_1}\left(2+(1-4\theta)\wh{a}_1h^2\frac{|\wh{a}|}{\wh{a}_1}\right),\ \
 1\leq\frac{|\wh{a}|}{\wh{a}_1}\leq \frac{|\wh{a}_0|}{\wh{a}_1}+1,\ \
 (1-4\theta)\wh{a}_1h^2\leq |\wt{\alpha}_\theta|\leq 2+(1-4\theta)\wh{a}_1h^2\frac{|\wh{a}|}{\wh{a}_1}.
\end{gather*}
According to Corollary \ref{cor:p7a}, this implies the result.
\end{proof}
\par Notice that if $|2+(1-4\theta)h^2\wh{a}_0|\geq\varepsilon_0>0$ (in particular, if $\theta=\frac14$, or $V_\infty\geq 0$, or $h$ is small enough), then conditions \eqref{eq:p7d1} are necessary and sufficient.
\par In practice, it is more effective to compute $R^m=R^m(\vk,\mu)$ by the recurrence relations \cite{EA01,DZZ09}
\begin{gather}\label{eq:p86a}
 R^0=1,\ \ R^1=-\vk\mu,\ \
 R^m=\frac{2m-3}{m}\,\vk\mu R^{m-1}-\frac{m-3}{m}\,\vk^2R^{m-2}\ \ \text{for}\ \ m\geq2.
\end{gather}
\begin{corollary}\label{cor:p8}
The operator $\mathcal{S}_{\rm ref\,1/4}$ (defined by formulas \eqref{eq:p53} and \eqref{eq:p54} for $\theta=\frac{1}{4}$) is independent of $h$ since its parameters are
\begin{gather}
 c_{0\,1/4}=-\left(\frac{|\wh{a}|}{2}\right)^{1/2}e^{-i\,(\arg_0 \wh{a})/2},\ \
 \vk_{1/4}=-e^{i\arg\wh{a}},\ \
 \mu_{1/4}=\frac{\wh{a}_0}{\wh{a}}.
\label{eq:p83}
\end{gather}
\end{corollary}
\begin{proof}
Clearly $\wh{\alpha}_{1/4}=2\,\wh{a}$ and $\wh{\beta}_{1/4}=2\,\wh{a}_0$ that implies the result.
\end{proof}
\par Notice that, in the particular case $V_{\infty}=0$, we get $\vk_{1/4}=-i$ and $\mu_{1/4}=0$, thus the formulas for $\mathcal{S}_{\rm ref\,1/4}$ are essentially simplified since the right formula \eqref{eq:p83} and the recurrence relations \eqref{eq:p86a} are reduced to
\begin{equation}
 c_{0\,1/4}=-\left(\frac{|\wh{a}_1|}{2}\right)^{1/2}e^{-i\,\pi/4}\ \ \text{and}\ \
 R^0=1,\ \ R^{2l-1}=0, \ \ R^{2l}=\frac{2l-3}{2l}R^{2(l-1)}\ \ \text{for}\ \ l\geq 1,
\label{sdtbckernel}
\end{equation}
so that $R^2=-\frac12$ and $R^{2l}=-\frac{(2l-3)!!}{(2l)!!}$ for $l\geq 2$.
\begin{remark}\label{rem:p8}
The stability bounds given for the family of schemes with the DTBC \eqref{eq:p31} in \cite{DZZ09} remain valid if one replaces $\mathcal{S}_{\rm ref\,\theta}$ by $\mathcal{S}_{\rm ref\,\theta_0}$ with any $\theta_0\leq\frac14$, in particular, by $\mathcal{S}_{\rm ref\,1/4}$.
\end{remark}
\par We also consider the semi-discrete Crank-Nicolson method for problem \eqref{eq:se}-\eqref{eq:s21}
\begin{gather}
 i\hbar\rho \ov{\pa}_t\Psi=-\frac{\hbar^{\,2}}{2}D(B D\ov{s}_t\Psi)+V\ov{s}_t\Psi\ \
 \text{on}\ \ \mathbb{R}^+\times\omega^{\tau},
\label{eq:p91}\\
%[1mm]
 \Psi|_{x=0}=0,\ \ \Psi^0=\psi_0\ \ \text{on}\ \ \mathbb{R}^+,
\label{eq:p92}
\end{gather}
where $\Psi$ is defined on $\overline{\mathbb{R}}^+\times\ov{\omega}^{\,\tau}$ and $\Psi^m(x)\to0$ as $x\to+\infty$ for any $m\geq1$.
\par We write down the corresponding SDTBC allowing to restrict its solution to $[0,X]\times\ov{\omega}^{\,\tau}$ in the form
\begin{equation}
 (\left.D\ov{s}_t\Psi)\right|_{x=X} = \mathcal{S}_{D}\bPsi_{X}\equiv c_{0D}R_D*\Psi_{X}\ \ \text{on}\ \ \omega^{\tau}
\label{eq:p92}
\end{equation}
similar to \eqref{eq:p31}, where $R^0_D=1$ and $\Psi_{X}=\left.\Psi\right|_{x=X}$.
The SDTBCs were previously considered in the slightly different form
$(\left.D\Psi)\right|_{x=X} = \wt{\mathcal{S}}_{D}\bPsi_{X}$ on $\omega^{\tau}$
or in the form of the corresponding Neumann-to-Dirichlet map in \cite{SD95,SY97,YFS01,AB03,AABES08}. Note that they were presented there explicitly only in the particular case $V_{\infty}=0$ when they were simplified essentially.
\par The next result is somewhat unexpected.
\begin{proposition}\label{prop:p10}
The operators $\mathcal{S}_{D}$ and $\mathcal{S}_{\rm ref\,1/4}$
coincide.
\end{proposition}
\begin{proof}
We derive the operator $\mathcal{S}_{D}$ by the approach from \cite{DZ06,DZZ09} to clarify the result. For brevity, we confine ourselves by a formal derivation.
We recall the reproducing function
\[
 r(z)=\mathcal{T}[R](z):=\sum_{m=0}^{\infty}R^mz^m,\ \ z\in\mathbb{C},
\]
of $R$: $\ov{\omega}^{\,\tau}\to \mathbb{C}$ and the inverse transform $R=\mathcal{T}^{-1}[r]$ defined by
$
 R^m=\frac{r^{(m)}(0)}{m!},\ \ m\geq0.
$
\par For $x\geq X_0$, equation \eqref{eq:p91} takes the simpler form
\begin{equation}\label{eq:p101}
 i\hbar\rho_{\infty} \ov{\pa}_t\Psi=-\frac{\hbar^{\,2}}{2}B_{\infty}D^2\ov{s}_t\Psi+V_{\infty}\ov{s}_t\Psi
\end{equation}
(cp. to \eqref{eq:sec}) and also $\Psi^0(x)=0$. Applying the operator $\mathcal{T}$ to this equation with constant coefficients, we get the second order ODE in $x$ with the complex parameter $z$
\[
 i\hbar\rho_{\infty} \frac{1-z}{\tau}\wt{\Psi}+\frac{1+z}{2}\left(\frac{\hbar^{\,2}}{2}B_{\infty}D^2\wt{\Psi}-V_{\infty}\wt{\Psi}\right)=0
\]
for the function $\wt{\Psi}(x,z):=\mathcal{T}[\Psi(x)](z)$. We rewrite it in the canonical form
\[
 D^2\wt{\Psi}-\lambda(z)\wt{\Psi}=0\ \ \text{with}\ \ \lambda(z)=2\frac{az+a^*}{z+1}.
\]
\par Its solution such that $\wt{\Psi}(x,z)\to\infty$ as $x\to+\infty$ has the form
\[
 \wt{\Psi}(x,z)=\wt{\Psi}(X,z)\exp\left\{\sqrt[(-)]{\lambda(z)}(x-X)\right\}\ \ \text{for}\ \ x\geq X,
\]
where $\sqrt[(-)]{\cdot}$ is the branch of $\sqrt{\cdot}$ with the negative real part. Then
\[
 (D\ov{s}_t\wt{\Psi})(X,z) = \frac{1+z}{2}\sqrt[(-)]{\lambda(z)}\,\wt{\Psi}(X,z)
\]
and according to \eqref{eq:p92} consequently
\[
 \mathcal{S}_{D}\bPsi_{X}=\mathcal{T}^{-1}\left[\frac{1+z}{2}\sqrt[(-)]{\lambda(z)}\,\wt{\Psi}(X,z)\right]=c_{0D}R_D*\Psi_{X}
 \ \ \text{with}\ \
 c_{0D}=\mathcal{T}^{-1}\left[\frac{1+z}{2}\sqrt[(-)]{\lambda(z)}\right].
\]
\par Similarly to \cite{DZZ09}, for $z$ small enough, we have
\[
 \frac{1+z}{2}\sqrt[(-)]{\lambda(z)}=\frac{1}{2}\sqrt[(-)]{\lambda(0)}\,\sqrt[+]{\frac{a}{a^*}z^2+2\frac{a_0}{a^*}z+1},
\]
where $\sqrt[+]{\cdot}$ is the analytic branch of $\sqrt{\cdot}$ in the disk $\{\abs{z-1}<1\}$ such that $\sqrt[+]{1}=1$. Clearly
\[
 c_{0D}=\frac{1}{2}\sqrt[(-)]{\lambda(0)}=-\left(\frac{\abs{\wh{a}}}{2}\right)^{1/2}e^{-i(\arg_0\wh{a})/2}=c_{0\, 1/4},
\]
see \eqref{eq:p83}, and
\[
 R_D=\mathcal{T}^{-1}\left[\sqrt[+]{(\vk_{1/4}z)^2-2\mu_{1/4}\vk_{1/4}z+1}\,\right]=R(\vk_{1/4}, \mu_{1/4})
\]
according to \cite{DZZ09} since $\frac{a}{a^*}=\vk^2_{1/4}$ and $\frac{a_0}{a^*}=-\vk_{1/4}\mu_{1/4}$. The proof is complete.
\end{proof}
\par Now we can study closeness for the kernels of the operators $\mathcal{S}_{\rm ref\,\theta}$ and $\mathcal{S}_{D}$.
\begin{proposition}\label{prop:p12}
The following bound holds
\begin{equation}\label{eq:p121}
 \abs{c_{0\theta}R^m(\vk_{\theta},\mu_{\theta})-c_{0D}R_D^m}\leq
 \left(\frac{3\sqrt{2}}{\abs{\wt{\alpha}_{\theta}}}
 +\frac{1}{\abs{2m-1}(\abs{\wt{\alpha}_\theta}^{1/2}+\sqrt{2})}\right)
 (1-4\theta)h^2\abs{\wh{a}}^{3/2}%\ \ \text{for}\ \ m\geq0,
\end{equation}
for $m\geq0$ (recall that $\wt{\alpha}_{\theta}=2+(1-4\theta)h^2\,\wh{a}$).
In particular
\begin{equation}\label{eq:p122}
 \sup_{m\geq0}\abs{c_{0\theta}R^m(\vk_{\theta},\mu_{\theta})-c_{0D}R_D^m}=O\left((1-4\theta)\frac{h^2}{\tau^{3/2}}\right)\ \
 \text{as}\ \ (1-4\theta)\frac{h^2}{\tau}\to 0\ \ \text{and}\ \ \tau\leq\tau_0.
\end{equation}
\end{proposition}
\begin{proof}
Clearly
\[
 c_{0\theta}R^m(\vk_{\theta},\mu_{\theta})=\frac{(-1)^m\abs{\wh{\alpha}_{\theta}}^{1/2}}{2(2m-1)}e^{i(m-1/2)\arg_0\wh{\alpha}_{\theta}}
 \left[P_m(\mu_{\theta})-P_{m-2}(\mu_{\theta})\right].
\]
Therefore
\begin{gather*}
 r^m_{\theta}:=\abs{c_{0\theta}R^m(\vk_{\theta},\mu_{\theta})-c_{0\,1/4}R^m(\vk_{1/4},\mu_{1/4})}\leq\frac{\rho^m_{\theta}}{2}\abs{P_m(\mu_{\theta})-P_{m-2}(\mu_{\theta})}
\\
 +\frac{|\wh{\alpha}_{1/4}|^{1/2}}{2}
 \abs{\frac{1}{2m-1}\left[P_m(\mu_{\theta})-P_{m-2}(\mu_{\theta})\right]
     -\frac{1}{2m-1}\left[P_m(\mu_{1/4})-P_{m-2}(\mu_{1/4})\right]}
\end{gather*}
with
\[
 \rho^m_{\theta}:=\frac{1}{\abs{2m-1}}\abs{
    \abs{\wh{\alpha}_{\theta}}^{1/2}e^{i(m-1/2)\arg_0\wh{\alpha}_{\theta}}
   -\abs{\wh{\alpha}_{1/4}}^{1/2}e^{i(m-1/2)\arg_0\wh{\alpha}_{1/4}}
 }.
\]
\par The Legendre polynomials have the properties
\[
 \max_{[-1,1]}\abs{P_m(\mu)}\leq1,\ \ \frac{1}{2m-1}\left[P_m(\mu)-P_{m-2}(\mu)\right]'=P_{m-1}(\mu)\ \ \text{for}\ \ m\geq 0,
\]
for example, see \cite{BE53}.
Consequently
\begin{equation}\label{eq:p133}
 r^m_{\theta}\leq\rho^m_{\theta}+\frac{\abs{\wh{\alpha}_{1/4}}^{1/2}}{2}\abs{\mu_{\theta}-\mu_{1/4}}.
\end{equation}
Notice that owing to \eqref{eq:p7a1} and $\wh{\alpha}_{1/4}=2\wh{a}$ we get
\[
 \rho^m_{\theta}\leq\frac{1}{\abs{2m-1}}\left(
 \abs{2\wh{a}}^{1/2}\abs{e^{i(m-1/2)\arg_0\wt{\alpha}_{\theta}}-1}+\abs{\abs{\wh{\alpha}_{\theta}}^{1/2}-\abs{2\wh{a}}^{1/2}}
 \right).
\]
Exploiting the relations
\[
 \abs{e^{i\lambda}-1}=2\abs{\sin\frac{\lambda}{2}},\ \
 \abs{\sin k\lambda}\leq\abs{k}\abs{\sin\lambda}\ \ \text{for real}\ \lambda\ \ \text{and integer}\ k,
\]
we further obtain
\begin{gather*}
 \rho^m_{\theta}\leq\frac{\abs{\wh{a}}^{1/2}}{\abs{2m-1}}\left(2\sqrt{2}\abs{2m-1}\sin\frac{\arg_0\wt{\alpha}_{\theta}}{4}
+\frac{\abs{\abs{\wt{\alpha}_{\theta}}-\abs{2}}}{\abs{\wt{\alpha}_{\theta}}^{1/2}+\sqrt{2}}
 \right)
\\
\leq\abs{\wh{a}}^{1/2}\left(\sqrt{2}\abs{e^{i\arg\wt{\alpha}_{\theta}}-1}
+\frac{(1-4\theta)h^2\abs{\wh{a}}}{\abs{2m-1}(\abs{\wt{\alpha}_{\theta}}^{1/2}+\sqrt{2})}
 \right).
\end{gather*}
Next we have
\begin{gather*}
 \abs{e^{i\arg\wt{\alpha}_{\theta}}-1}=\abs{\frac{\wt{\alpha}_{\theta}}{\abs{\wt{\alpha}_{\theta}}}-1}
\leq\frac{2(1-4\theta)h^2\abs{\wh{a}}}{\abs{\wt{\alpha}_{\theta}}},
\\
 \abs{\mu_{\theta}-\mu_{1/4}}
=\frac{\abs{\wh{a}_0(2-\abs{\wt{\alpha}_{\theta}})+(1-4\theta)h^2\abs{\wh{a}}^2}}{\abs{\wh{a}}\abs{\wt{\alpha}_{\theta}}}
\leq\frac{(1-4\theta)h^2(\abs{\wh{a}_0}+\abs{\wh{a}})}{\abs{\wt{\alpha}_{\theta}}}.
\end{gather*}
Using the last three bounds in \eqref{eq:p133} and recalling Proposition \ref{prop:p10}, we derive bound \eqref{eq:p121}. Also $\wh{a}=O(\frac{1}{\tau})$ and $\wt{\alpha}_{\theta}=2+o(1)$ as $(1-4\theta)\frac{h^2}{\tau}\to0$ and $\tau\leq\tau_0$ that implies \eqref{eq:p122}.
\end{proof}
According to \eqref{eq:p122}, in particular, $\sup_{m\geq0}\abs{c_{0\theta}R^m(\vk_{\theta},\mu_{\theta})-c_{0D}R_D^m}$ is of order $O(h^2)$ for fixed $\tau$ and $h\to0$, or tends to 0 as $\tau\to0$ and $h=o(\tau^{3/2})$.
\begin{remark}\label{rem:p15}
Notice that bound \eqref{eq:p121} is exact enough even for $m=0$ since
\[
 c_{0\theta}-c_{0D}\sim c_{0D}\frac{1-4\theta}{4}h^2\wh{a}^{\,*}\ \ \text{as}\ \ (1-4\theta)h^2|\wh{a}|\to 0
\]
(recall that $|c_{0D}|=\bigl(\frac{|\wh{a}|}{2}\bigr)^{1/2}$).
\end{remark}
\section{Numerical experiments}
\label{NE}
In this section we present the interesting results of numerical experiments on replacing the discrete convolution in time in the DTBC by the corresponding one from the SDTBC.
We consider the initial-boundary value problem \eqref{eq:se}-\eqref{eq:ic} for the Schr\"{o}dinger equation with the constant coefficients $\rho(x)\equiv 1$, $B(x)\equiv 2$, $V(x)\equiv 0$ and the scaled $\hbar=1$.
\par We also exploit the finite uniform meshes $x_j=jh$, $0\leq j\leq J$, with $h=\frac{X}{J}$ and $t_m=m\tau$, $0\leq m\leq M$, with $\tau=\frac{T}{M}$. To apply the SDTBC, we discretize \eqref{eq:p92} mainly similarly to \eqref{eq:p31} replacing ${\mathcal S}_{\rm ref\,\theta}$ by $\mathcal{S}_D$. But according to Remark \ref{rem1}, this reduces the total approximation order of the Numerov scheme, i.e., for $\theta=\frac{1}{12}$, so, in this case, below we also exploit the improved SDTBC (ISDTBC) combining the left-hand side of \eqref{eq:p31} for $\theta=\frac16$ together with its right-hand one for $\theta=\frac14$. Looking ahead, we will see that this change really improves the accuracy.
\par We rely upon the well-known exact solution
\[
 \psi(x,t)=\psi_G(x,t)\equiv \frac{1}{\sqrt[+]{1+i\,\frac{t}{\alpha}}}
 \exp\left\{ik(x-x^{(0)}-kt)-\frac{(x-x^{(0)}-2kt)^2}{4(\alpha+it)}\right\}
\]
(the Gaussian wave package), with the real parameters $k$ (the wave number), $\alpha>0$ and $x^{(0)}$. Then
\begin{equation}
\psi^0(x)=\psi_G(x,0)
=\exp\left\{ik(x-x^{(0)})-\frac{(x-x^{(0)})^2}{4\alpha}\right\}.
\label{s62}
\end{equation}
Though $\psi_G(0,t)$ and $\psi_G(x,0)$ are non-zero, below they both are small enough for any $t\geq 0$ and $x\geq X$.
\par We choose the parameters $k=100$ (that is rather high), $\alpha=\frac{1}{120}$ and $x^{(0)}=0.8$ together with $X=1.5$
and $T=0.006$ (taken in several previous papers including \cite{EA01,DZZ09,ZZ12}).
On Figure~\ref{fig:EX01:Initial+SolutionNorm} we give the modulus and the real part of the initial function and $L^2$-norm and $C$ (i.e., the uniform) one over $[0,X]$ of the solution in dependence with time.
The wave package is moving to the right and, for $T=0.006$, is almost leaving the computational domain, and thus the norms decrease abruptly.
\begin{figure}[htbp]
    \centering{
    \begin{minipage}[h]{0.49\linewidth}
        \center{\includegraphics[width=1\linewidth]{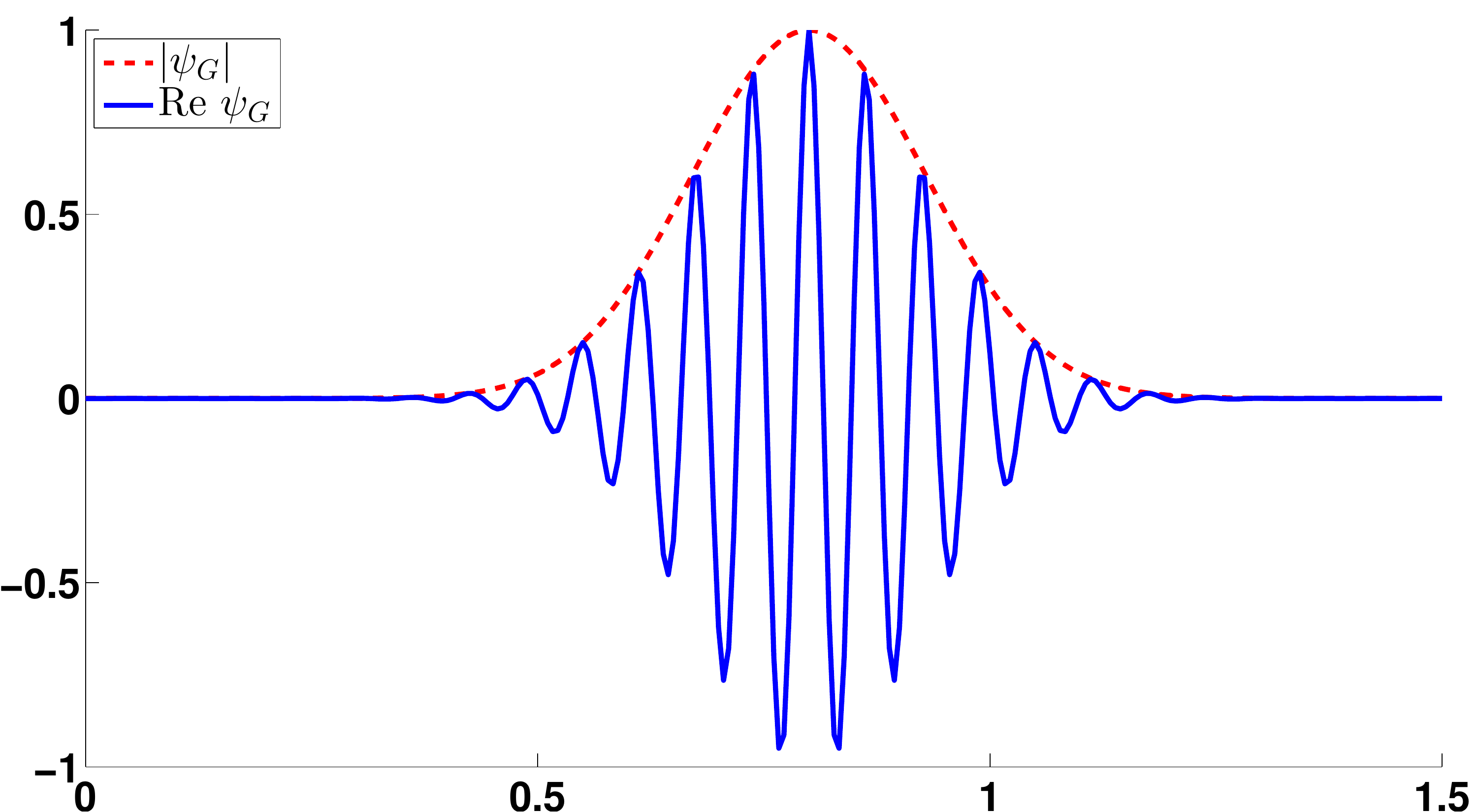}}
    \end{minipage}
    \ \ \ \ \ %\hfill
    \begin{minipage}[h]{0.29\linewidth}
        \center{\includegraphics[width=1\linewidth]{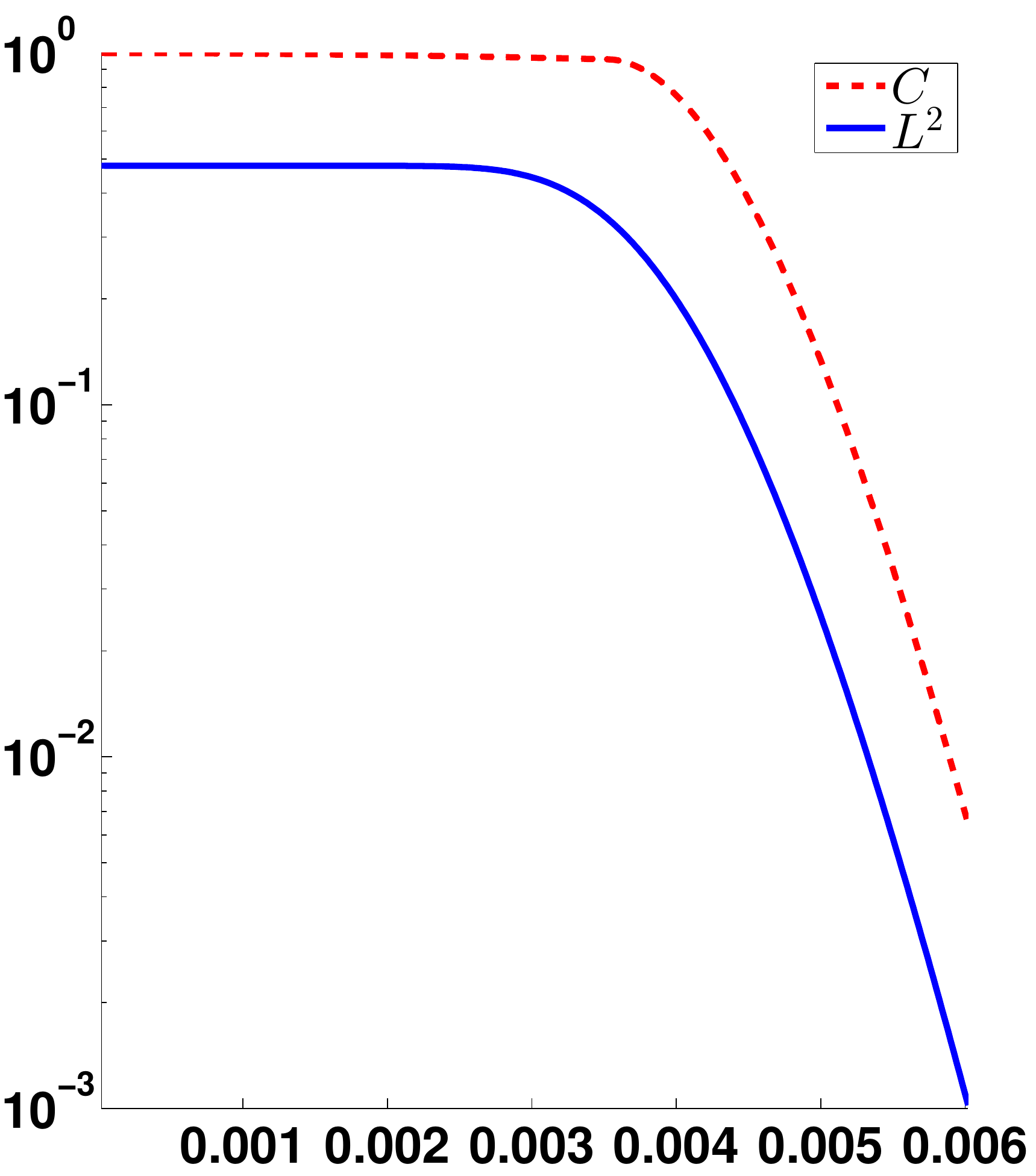}}
    \end{minipage}
    }
    \caption{\small{The modulus and the real part of the initial function $\psi_G(x,0)$ (left) and $L^2$ and $C$ norms of the solution $\psi_G$ in dependence with time (right)}}
    \label{fig:EX01:Initial+SolutionNorm}
\end{figure}
\par We compute the numerical solutions using the DTBC and the SDTBC for various $J$ and $M$ as well as $\theta$.
We first take $\theta=\frac{1}{12}$, $J=800$ and $M=3000$ and on Figure \ref{fig:EX01:DTBCvsSDTBC:Error} see that at the initial stage of  computing the behavior of both absolute and relative errors is the same in the DTBC and the SDTBC cases. But when the wave package is leaving the domain, in the DTBC case, the absolute errors decrease abruptly and the relative errors decrease slightly whereas, in the SDTBC case, the absolute errors stabilize and the relative errors increase significantly, reaching their high maximum values at the final computation moment $T$. The last behavior is rather typical. We emphasize that though both numerical solutions have reasonable absolute errors, the difference between the exploited discrete convolution kernels is significant that one clearly observes from Figure~\ref{fig:EX01:DTBCvsSDTBC:Error} where their modules are shown (notice carefully that, in the SDTBC kernel, zero elements for odd $m$, see \eqref{sdtbckernel}, are omitted). This means that an averaging effect plays the important role.
\begin{figure}[htbp]
    \centering{
    \begin{minipage}[h]{0.49\linewidth}
        \center{\includegraphics[width=1\linewidth]{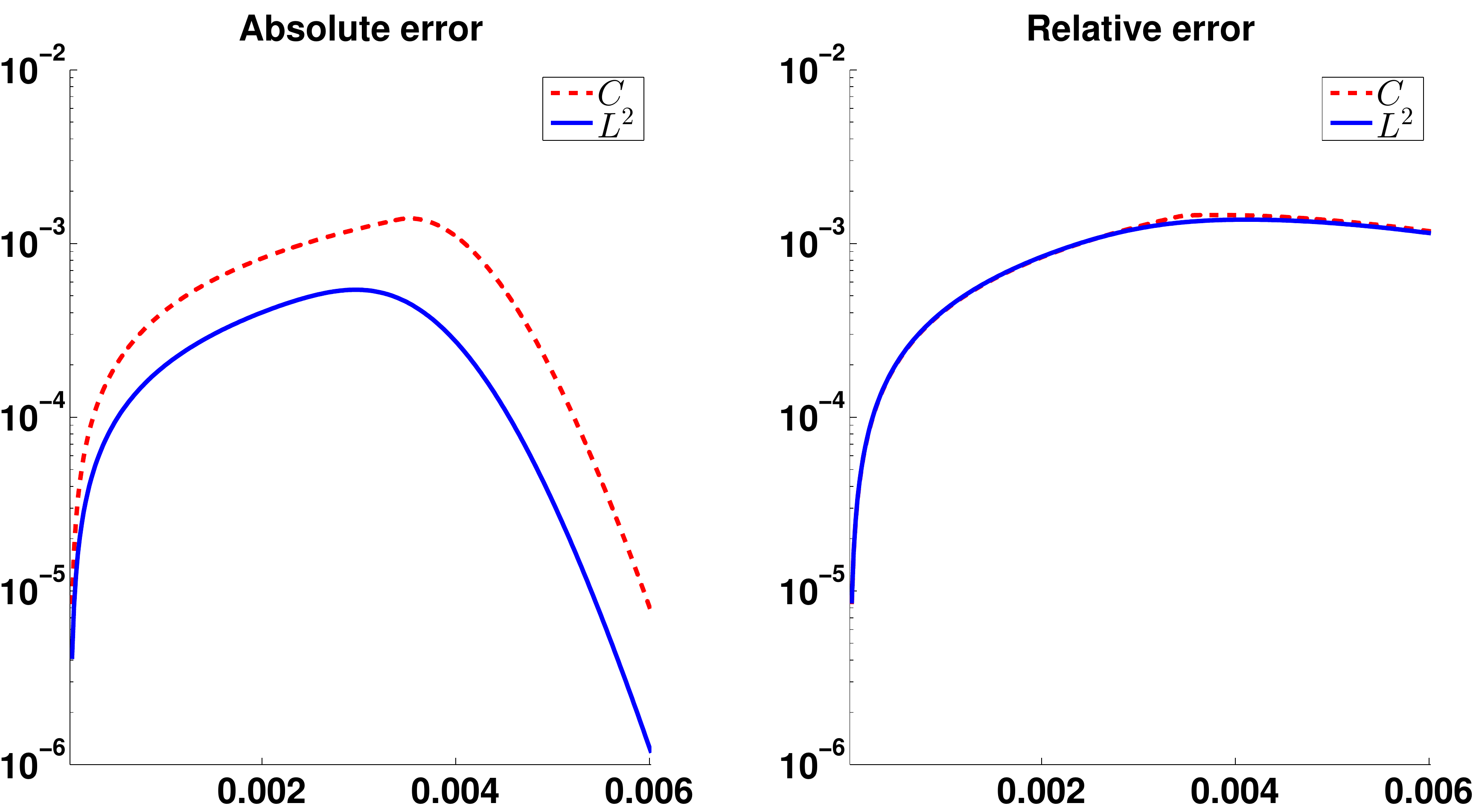}}
    \end{minipage}\\[3mm]
    \begin{minipage}[h]{0.49\linewidth}
        \center{\includegraphics[width=1\linewidth]{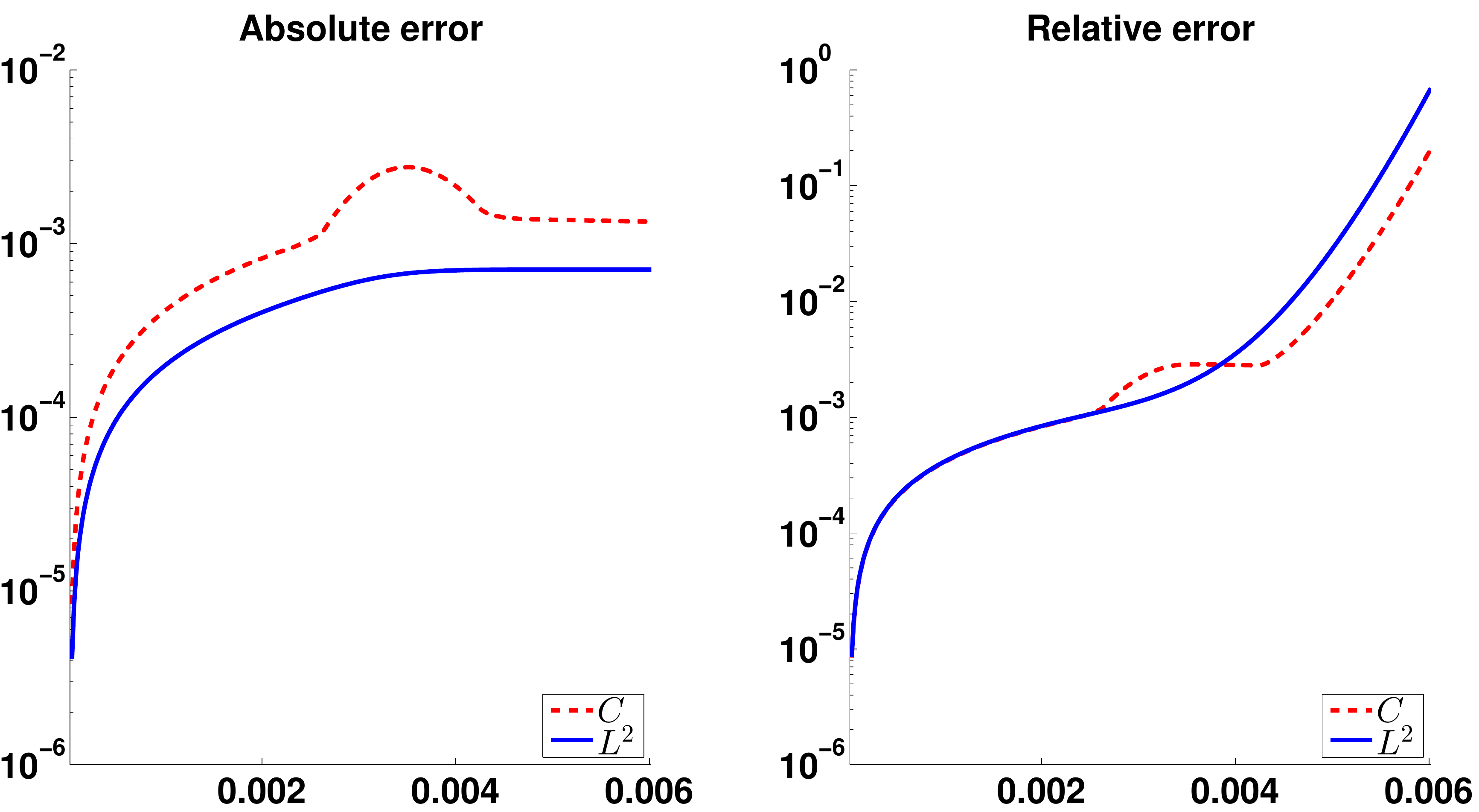}}\\
    \end{minipage}
    }
    \caption{\small The absolute and relative errors for the numerical solutions using the DTBC (upper) and the SDTBC (lower) for $\theta=\frac{1}{12}$ in dependence with time, for $J=800$ and $M=3000$}
    \label{fig:EX01:DTBCvsSDTBC:Error}
\end{figure}
\begin{figure}[htbp]
    \begin{minipage}[h]{0.49\linewidth}
        \center{\includegraphics[width=1\linewidth]{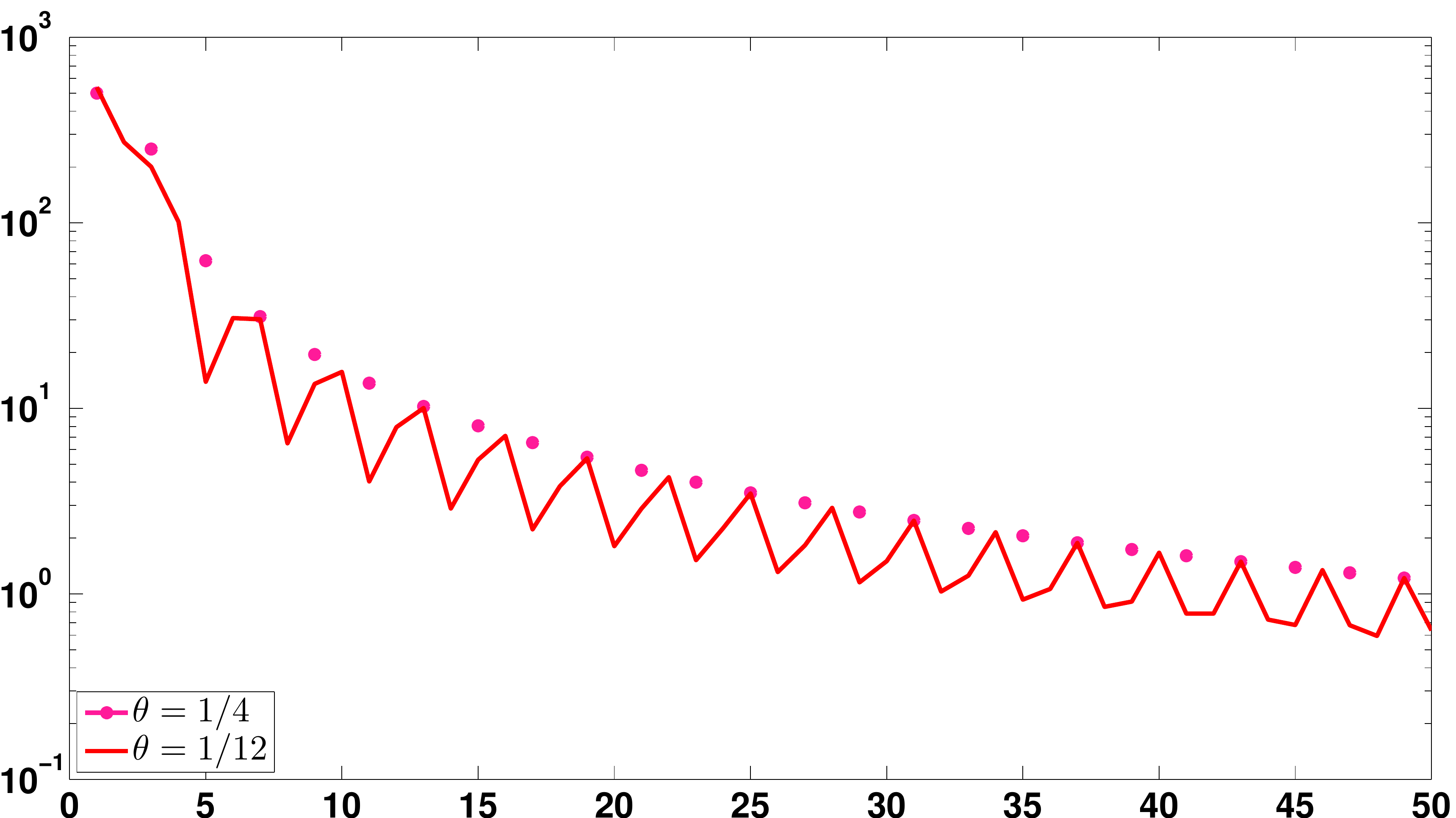}} \end{minipage}
    \begin{minipage}[h]{0.49\linewidth}
        \center{\includegraphics[width=1\linewidth]{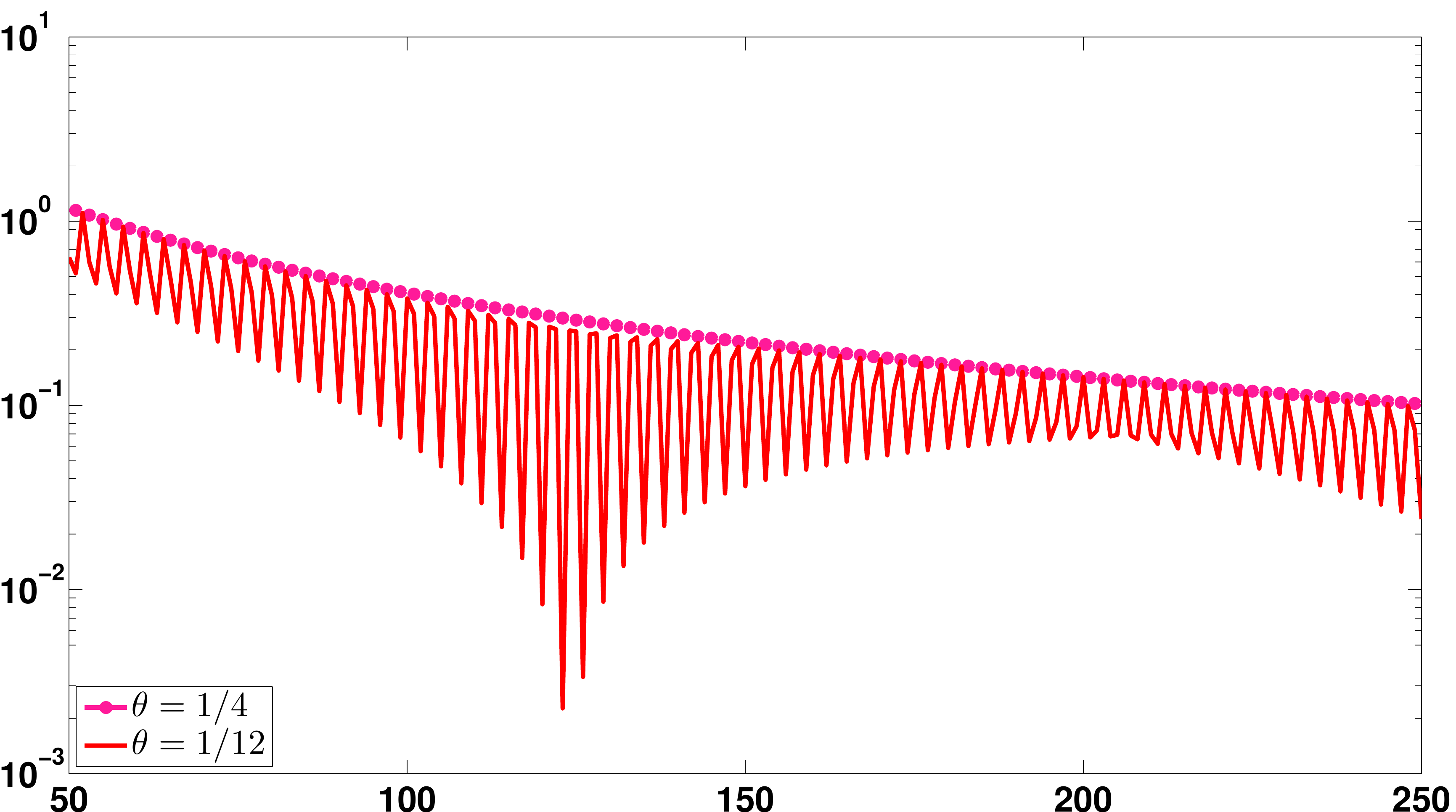}} \end{minipage}\
    \caption{\small The modules of the discrete convolution kernels $\abs{c_{0\theta}R_{\theta}^m}$, for $1\leq m\leq 50$ (left) and $50\leq m\leq 250$ (right), for $\theta=\frac{1}{12}$ and $\frac{1}{4}$ (in the latter case, zero elements for odd $m$ are omitted), and for $J=800$ and $M=3000$}
\label{fig:EX01:Kernel:M=3000}
\end{figure}
\par For $\theta=\frac{1}{12}$ and $M=6000$, in Table \ref{tab:EX01:DTBC:J} we present various errors for the numerical solutions using the DTBC (the upper table), the SDTBC (the middle table) and the ISDTBC (the lower table): the absolute maximum in time $L^2$-errors $E_{L^2}$, the absolute maximum in time $C$-errors $E_C$ and the associated maximum in time relative errors $E_{L^2,{\rm rel}}$ and $E_{C,{\rm rel}}$ together with their ratios as $J$ increases.
Comparing the results in the DTBC and the SDTBC cases, the latter absolute errors are higher but at the same level whereas the latter relative errors are much more higher.
In the DTBC case, for moderate values $J=400$ and $800$, we fix higher orders of decreasing for both the absolute and relative errors (notice that $R_{L^2}>6$ and $R_C>6$ there) whereas, in the case of the SDTBC, we can do that only for the absolute errors, moreover, for $J=800$, only $R_C$ is close to 5.
Also in the DTBC case, for larger values $J=1600$ and $1200$, the error decreasing orders become low because the value of $M$ is not sufficiently large. In the SDTBC case, the absolute $L^2$-error decreasing order is close to 2 (since $R_{L^2}\approx 4$) for $J=800$ and $1600$, while the relative error decreasing orders are very close to 2 for any $J$.
Notice in addition that the absolute and relative differences of the numerical solutions using the DTBC and the SDTBC all demonstrate the second decreasing order (we omit the corresponding table).
\par Passing to the ISDTBC clearly improves the absolute errors almost to their values in the DTBC case and remarkably improves the relative errors demonstrating their higher decreasing order close to 3 now (clearly bringing us to Remark \ref{rem1} once again).
\par On Figure \ref{fig:EX01:DTBC:MaxAbsError:M=3000}, we give the maximum in time \textit{absolute} $L^2$ and $C$ errors for various $\theta$ in dependence with $J=200,400,800,1600$ and $3200$, for $M=3000$.
For $\theta=0,\frac{1}{6},\frac{1}{4}$, the results are close in both the DTBC and the SDTBC cases, and the errors are maximal for $\theta=\frac{1}{4}$ whereas they are very close for $\theta=0$ and $\frac{1}{6}$ (except $J=3200$).
For $\theta=\frac{1}{12}$, the errors are significantly smaller than for the previous values of $\theta$, and in the DTBC case they are smaller compared to the SDTBC one. But passing to the ISDTBC makes the last mentioned errors very close to the DTBC case.
\begin{figure}[htbp]
    \begin{minipage}[h]{0.5\linewidth}\center{
        \includegraphics[width=1\linewidth]{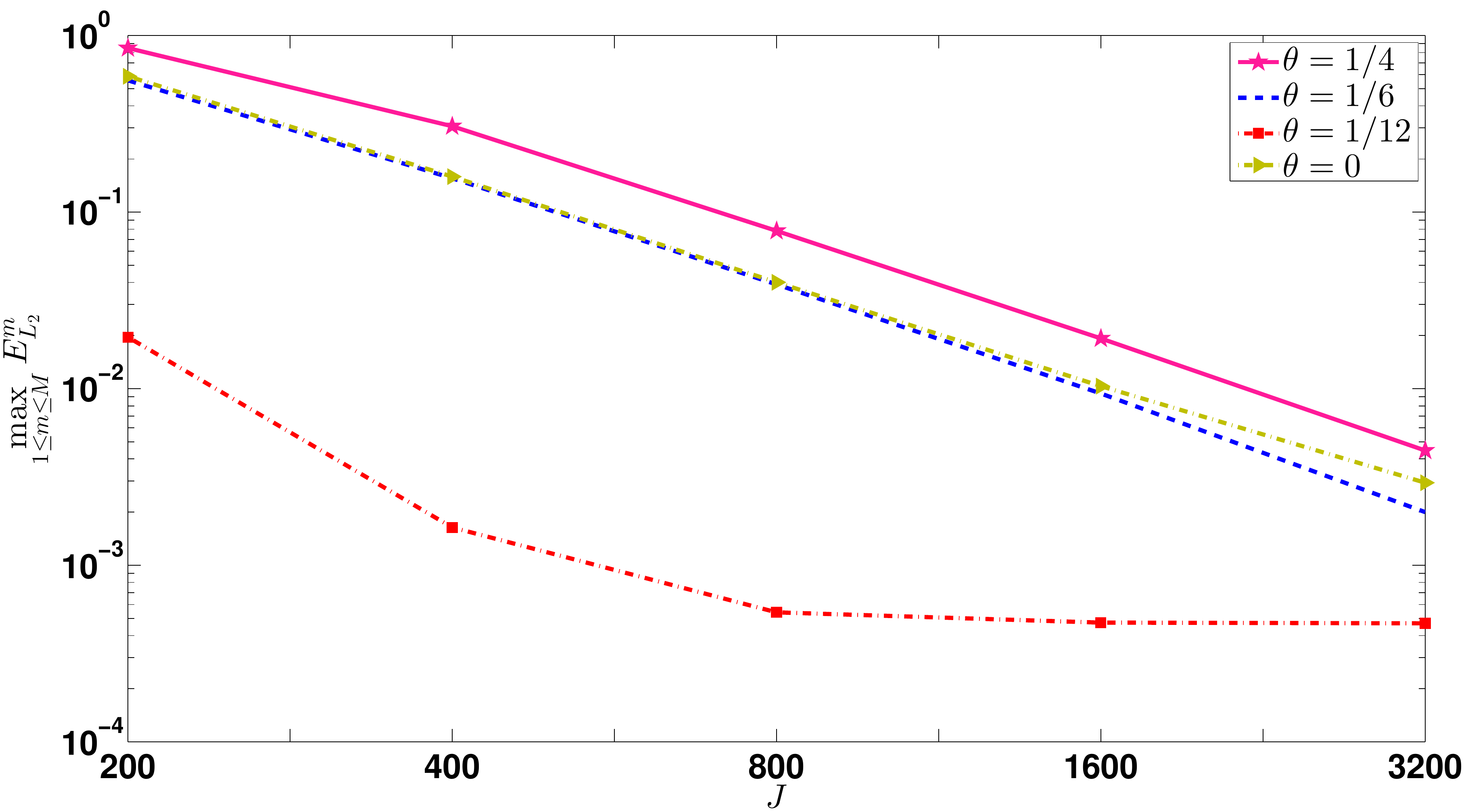}
    } \footnotesize{(a) in $L^2$-norm (DTBC)} \\
    \end{minipage}\hfill
    \begin{minipage}[h]{0.5\linewidth}\center{
        \includegraphics[width=1\linewidth]{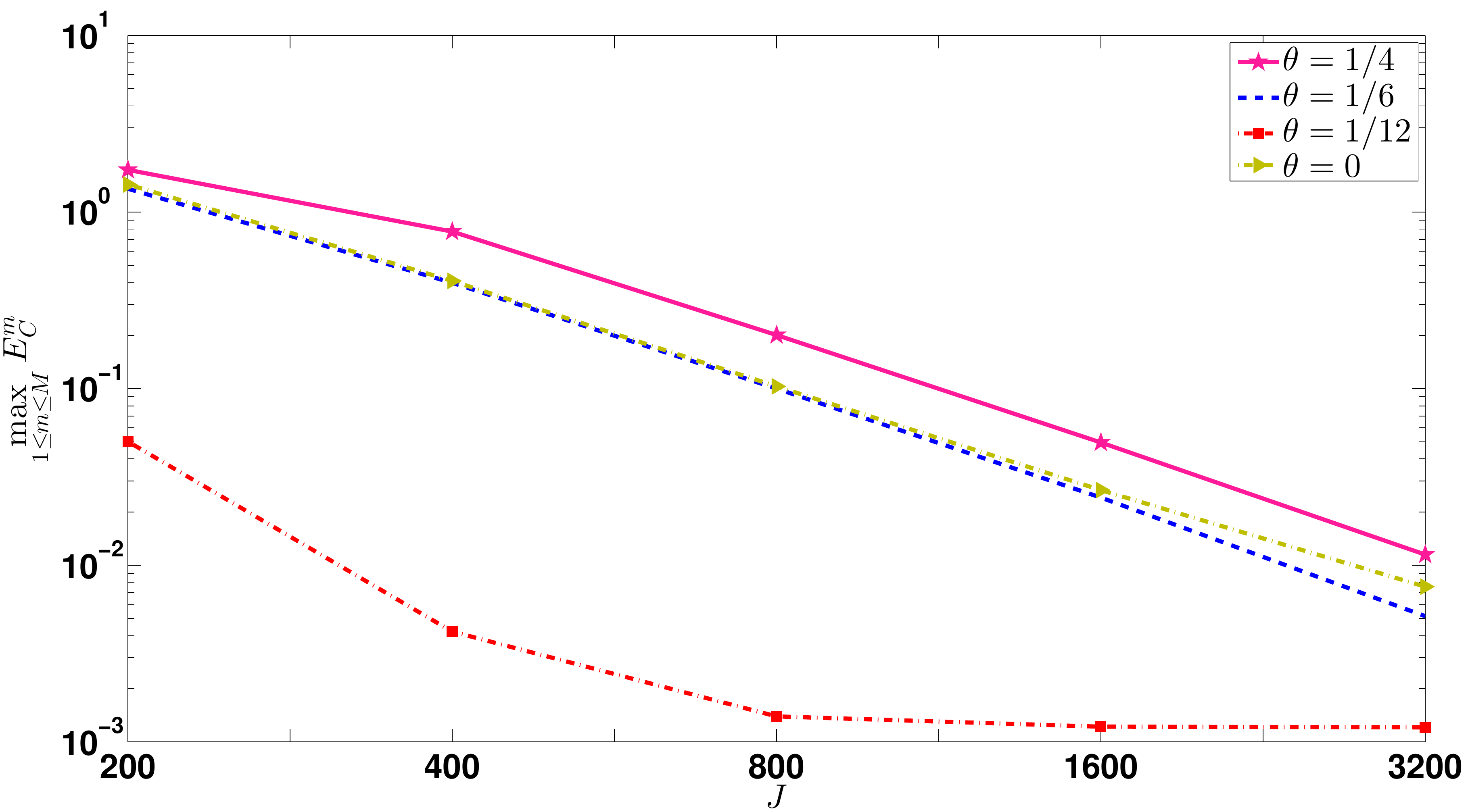}
    } \footnotesize{(b) in $C$-norm (DTBC)} \\
    \end{minipage}
    \begin{minipage}[h]{0.5\linewidth}\center{
        \includegraphics[width=1\linewidth]{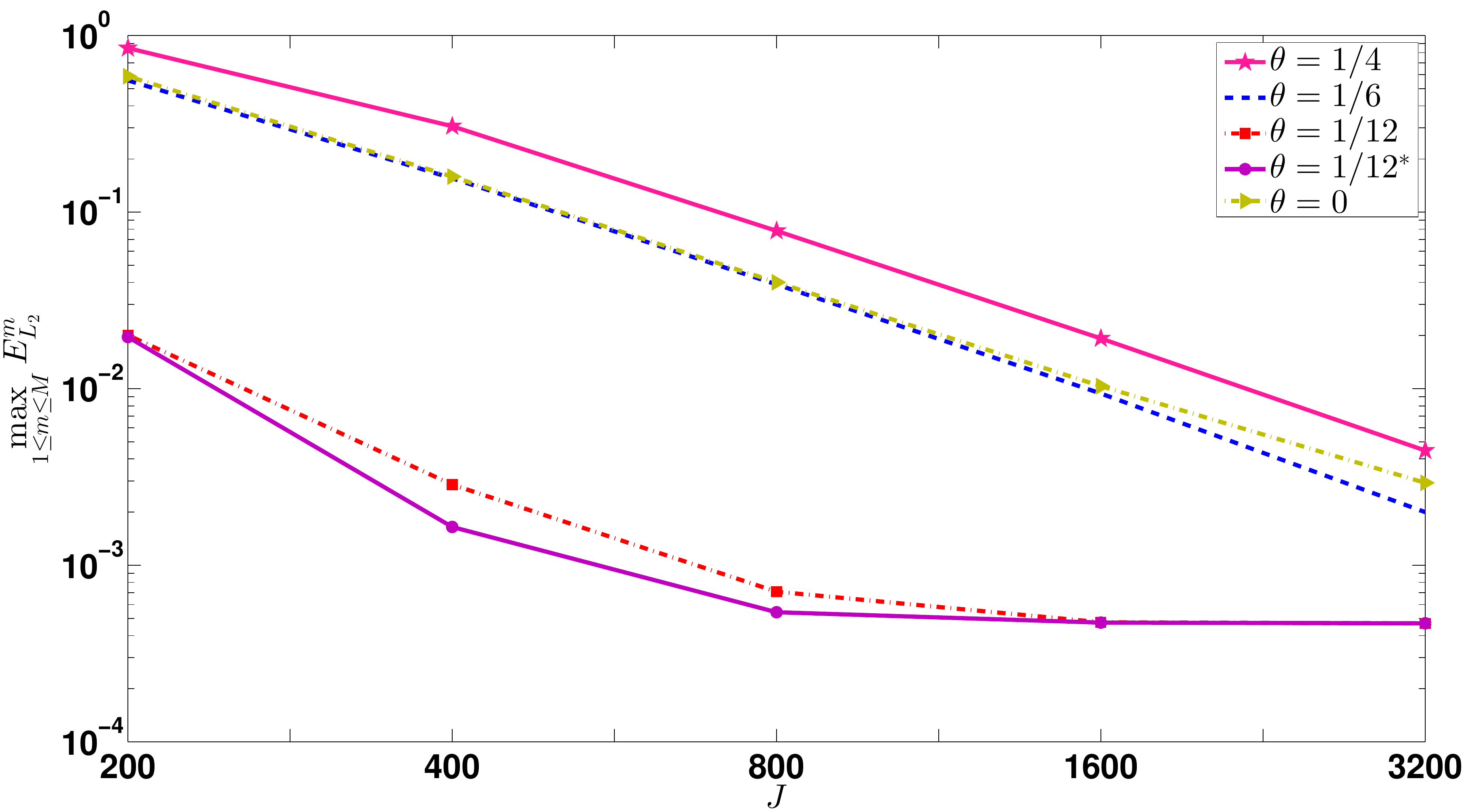} } \footnotesize{(c) in $L^2$-norm (SDTBC)} \\
    \end{minipage}\hfill
    \begin{minipage}[h]{0.5\linewidth}\center{
        \includegraphics[width=1\linewidth]{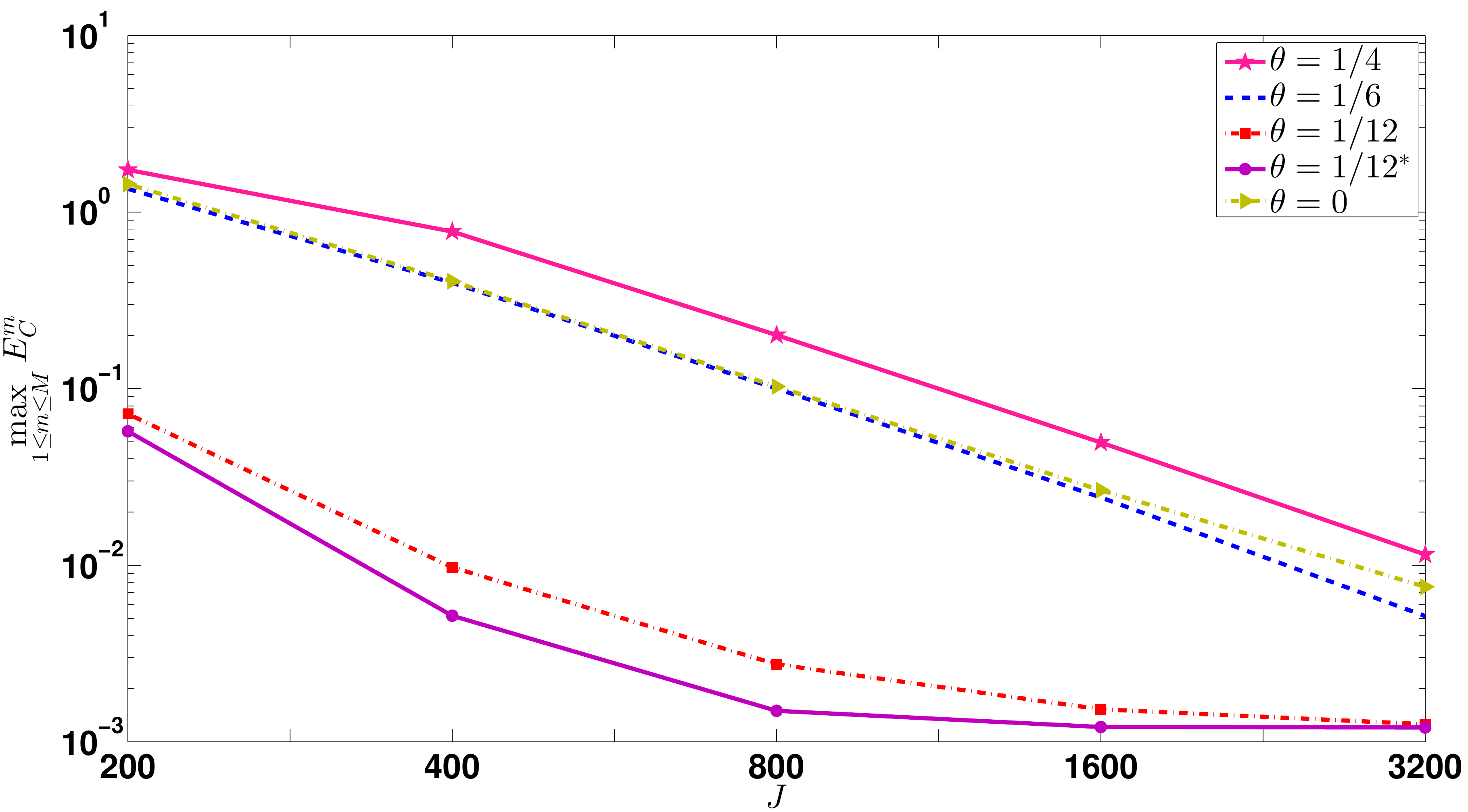}
         } \footnotesize{(d) in $C$-norm (SDTBC)} \\
    \end{minipage}
\caption{\small The maximum in time absolute errors in $L^2$ и $C$ norms for the numerical solutions using the DTBC (upper) and the SDTBC (lower) $\theta=0,\frac{1}{12},\frac{1}{6},\frac{1}{4}$ and $\frac{1}{12}^*$ (corresponding to $\theta=\frac{1}{12}$ and the ISDTBC), in dependence with $J=200,400,800,1600$ and $3200$, for $M=3000$}
\label{fig:EX01:DTBC:MaxAbsError:M=3000}
\end{figure}
\par On Figure \ref{fig:EX01:DTBC:MaxRelError:M=3000}, we show the corresponding maximum in time \textit{relative} $L^2$ and $C$ errors for the same $\theta$ in dependence with the same values of $J$, once again for $M=3000$.
In the DTBC case, the behavior of the relative errors and absolute ones is close (except the minimal $J=200$).
But in the SDTBC case, the situation is different. Namely, the scheme for $\theta=\frac{1}{12}$ loses its advantages and is no more the best in either $L^2$-norm or $C$-one. The relative $L^2$ errors decrease strictly as $\theta$ increases.
The relative $C$ error is the largest also for $\theta=0$ but the smallest for $\theta=\frac{1}{6}$ now whereas the similar errors for $\theta=\frac{1}{12}$ and $\theta=\frac{1}{4}$ are located between them and are very close to each other for $J\geq 400$.
Once again passing to the ISDTBC reduces the relative error in $L^2$-norm and especially in $C$-norm significantly and makes the scheme for $\theta=\frac{1}{12}$ the best one.
\par Clearly on both Figures \ref{fig:EX01:DTBC:MaxAbsError:M=3000} and \ref{fig:EX01:DTBC:MaxRelError:M=3000}, the upper and lower graphs for $\theta=\frac{1}{4}$ are the same since the DTBC and the SDTBC coincide in this case.
\begin{figure}[htbp]
    \begin{minipage}[h]{0.5\linewidth}\center{
        \includegraphics[width=1\linewidth]{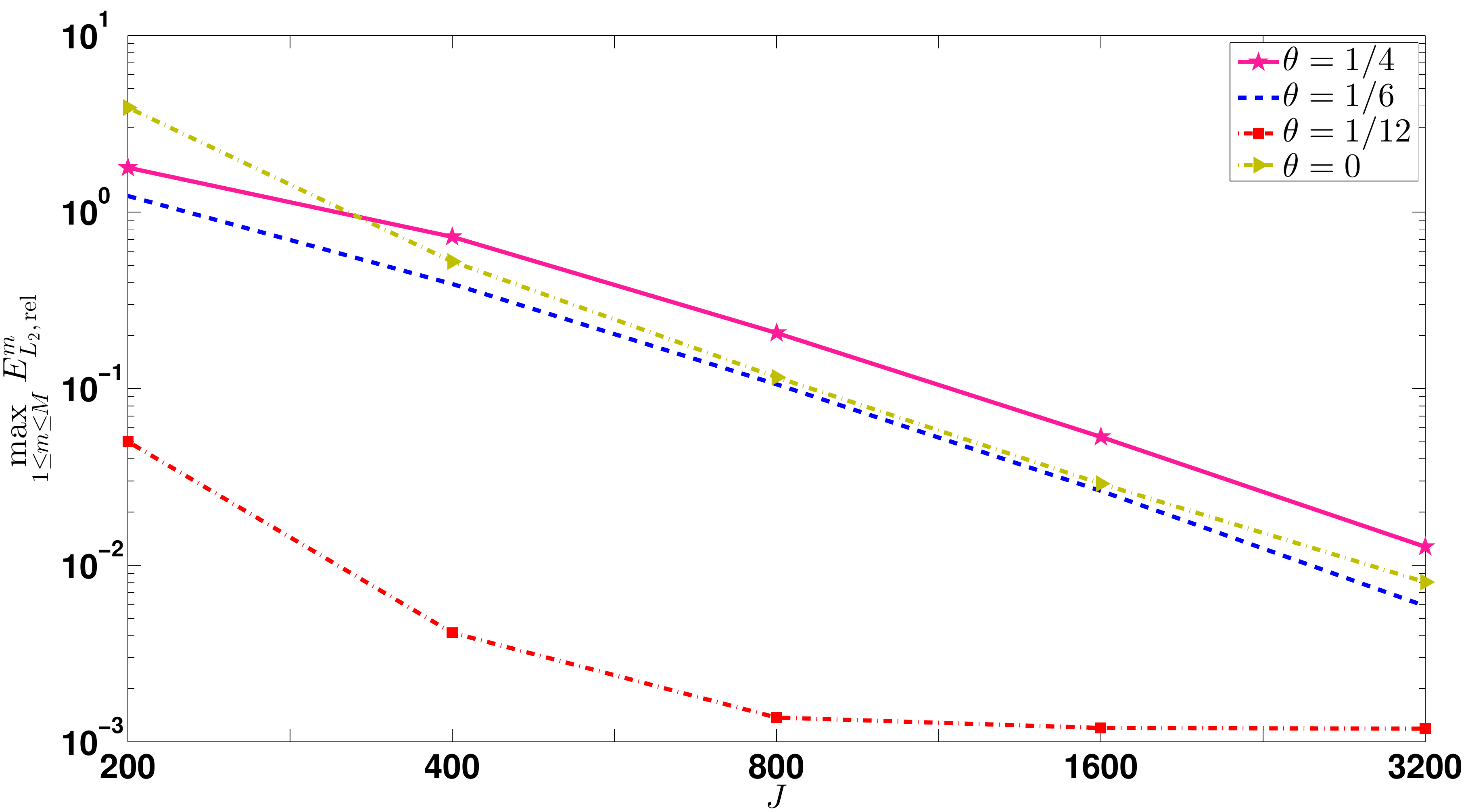}
    } \footnotesize{(a) in $L^2$-norm (DTBC)} \\
    \end{minipage}\hfill
    \begin{minipage}[h]{0.5\linewidth}\center{
        \includegraphics[width=1\linewidth]{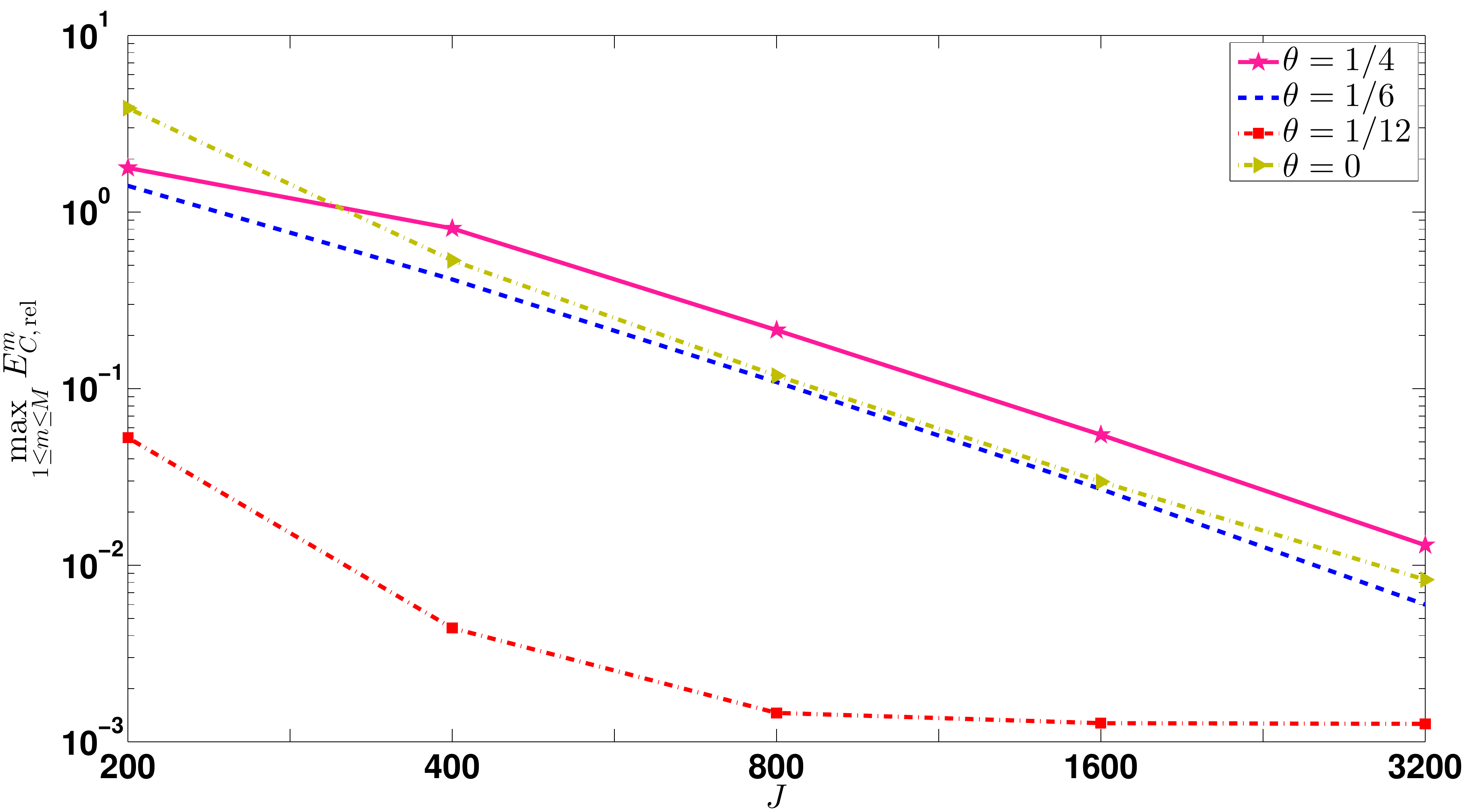}
    } \footnotesize{(b) in $C$-norm (DTBC)} \\
    \end{minipage}
    \begin{minipage}[h]{0.5\linewidth}\center{
        \includegraphics[width=1\linewidth]{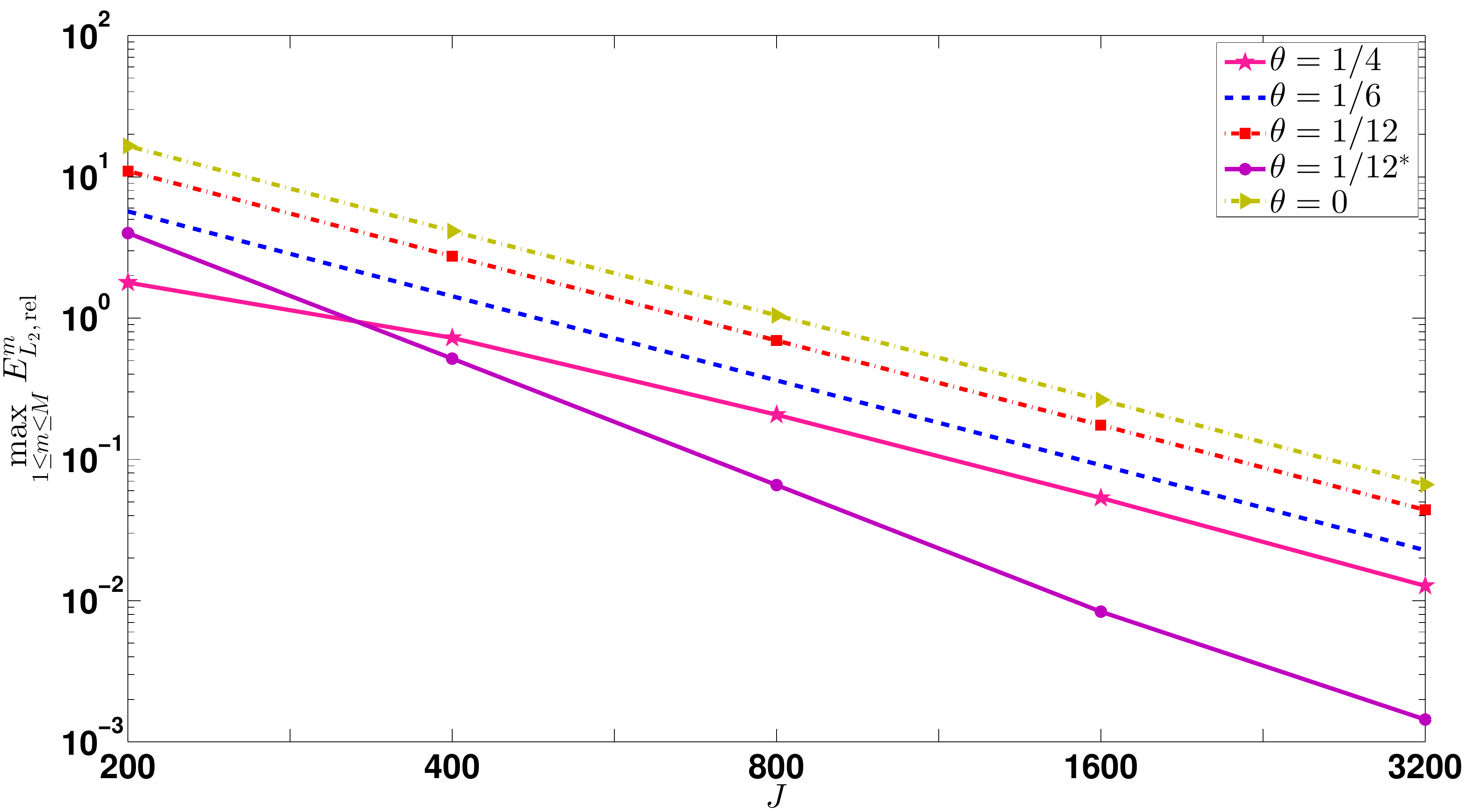}
            } \footnotesize{(c) in $L^2$-norm (SDTBC)} \\
    \end{minipage}\hfill
    \begin{minipage}[h]{0.5\linewidth}\center{
        \includegraphics[width=1\linewidth]{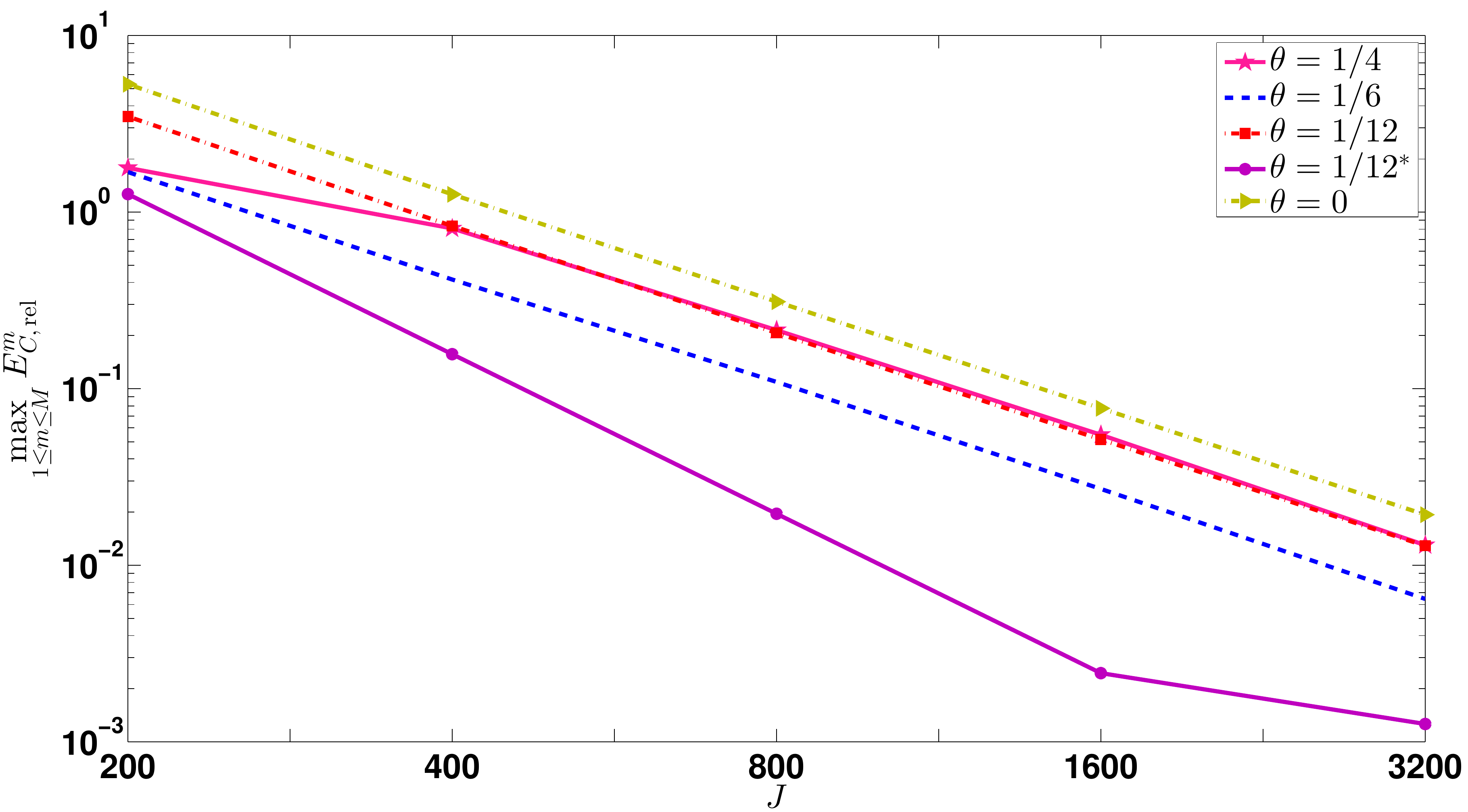} % 20131216_153211_1D_EX01_MaxError_rel_C_PARAM=1-4,1-6,1-12,0_M=3000_FREQ=10_C
    } \footnotesize{(d) in $C$-norm (SDTBC)} \\
    \end{minipage}
\caption{\small The maximum in time relative errors in $L^2$ и $C$ for the numerical solutions using the DTBC (upper) and the SDTBC (lower) for $\theta=0,\frac{1}{12},\frac{1}{6},\frac{1}{4}$ and $\frac{1}{12}^*$ (corresponding to $\theta=\frac{1}{12}$ and the ISDTBC), in dependence with $J=200,400,800,1600$ and $3200$, for $M=3000$}
\label{fig:EX01:DTBC:MaxRelError:M=3000}
\end{figure}
\par Finally, for $\theta=\frac{1}{12}$ and $J=3200$, in Table \ref{tab:EX01:DTBC:M} we present various errors for the numerical solutions using the DTBC and the SDTBC as $M$ increases.
In the DTBC case (the upper table), both the absolute and relative error decreasing orders are very close to 2.
But in the SDTBC case (the middle table), only the absolute error decreasing orders are close to 2 whereas the relative errors slightly decrease only for moderate values of $M$ and then remain almost unchanged. In the ISDTBC case (the lower table), all the errors are very close to the DTBC one (except the last two values of $E_{L^2,{\rm rel}}$).
\par Comparing the last results in the DTBC and the SDTBC cases, we have also found that the maximal absolute differences between the corresponding numerical solutions are less then $4.42\cdot10^{-5}$ in $L^2$-norm and $8.87\cdot10^{-5}$ in $C$-norm for all values of $M$ in the last table. Thus they are notably smaller than the absolute errors of both solutions, i.e., the numerical solutions are much closer to each other than to the exact one.
\par In addition, for the selected $J$, notice that the runtime is practically proportional to $M$ since the corresponding ratios of runtimes equal 1.98, 2.04, 1.92, 2.02 and 1.95, 2, 2.08, 2 respectively in the DTBC and the SDTBC cases. Thus the total cost for single computing the DTBC kernel and $M$-multiple computing the discrete convolutions in the DTBC or the SDTBC is inessential with respect to the cost for solving the linear algebraic systems in computing the numerical solution at all $M$ time levels.
\par The above and some other accomplished numerical experiments (involving non-zero potential $V$) demonstrate nice absolute error properties of the SDTBC. Moreover, they show that the closeness of the DTBC and the SDTBC deserves to be studied in more detail.

\section*{Acknowledgments}

The study is supported by The National Research University – Higher School of Economics' Academic Fund Program in 2014-2015, research grant No. 14-01-0014 (for the first author) and by the Russian Foundation for Basic Research, project No. 14-01-90009-Bel (for the second one).

\bigskip
%\par Received 16/01/2014.
\par\it{E-mail address}: {azlotnik2008@gmail.com}
\par\it{E-mail address}: {ilya.zlotnik@gmail.com}
\newpage
\begin{table}
\centering{
    \begin {tabular}{r<{\pgfplotstableresetcolortbloverhangright }@{}l<{\pgfplotstableresetcolortbloverhangleft }r<{\pgfplotstableresetcolortbloverhangright }@{}l<{\pgfplotstableresetcolortbloverhangleft }r<{\pgfplotstableresetcolortbloverhangright }@{}l<{\pgfplotstableresetcolortbloverhangleft }r<{\pgfplotstableresetcolortbloverhangright }@{}l<{\pgfplotstableresetcolortbloverhangleft }r<{\pgfplotstableresetcolortbloverhangright }@{}l<{\pgfplotstableresetcolortbloverhangleft }r<{\pgfplotstableresetcolortbloverhangright }@{}l<{\pgfplotstableresetcolortbloverhangleft }r<{\pgfplotstableresetcolortbloverhangright }@{}l<{\pgfplotstableresetcolortbloverhangleft }r<{\pgfplotstableresetcolortbloverhangright }@{}l<{\pgfplotstableresetcolortbloverhangleft }r<{\pgfplotstableresetcolortbloverhangright }@{}l<{\pgfplotstableresetcolortbloverhangleft }}%
\toprule \multicolumn {2}{c}{$J$}&\multicolumn {2}{c}{$E_{L^2}$}&\multicolumn {2}{c}{$R_{L^2}$}&\multicolumn {2}{c}{$E_C$}&\multicolumn {2}{c}{$R_{C}$}&\multicolumn {2}{c}{$E_{L^2,\,{\rm rel}}$}&\multicolumn {2}{c}{$R_{L^2,\,{\rm rel}}$}&\multicolumn {2}{c}{$E_{C,\,{\rm rel}}$}&\multicolumn {2}{c}{$R_{C,\,{\rm rel}}$}\\\midrule %
$200$&$$&$1$&$.92\cdot 10^{-2}$&--&&$4$&$.93\cdot 10^{-2}$&--&&$4$&$.92\cdot 10^{-2}$&--&&$5$&$.19\cdot 10^{-2}$&--&\\%
$400$&$$&$1$&$.29\cdot 10^{-3}$&$14$&$.95$&$3$&$.31\cdot 10^{-3}$&$14$&$.91$&$3$&$.26\cdot 10^{-3}$&$15$&$.1$&$3$&$.46\cdot 10^{-3}$&$14$&$.99$\\%
$800$&$$&$1$&$.90\cdot 10^{-4}$&$6$&$.77$&$4$&$.89\cdot 10^{-4}$&$6$&$.76$&$4$&$.81\cdot 10^{-4}$&$6$&$.77$&$5$&$.11\cdot 10^{-4}$&$6$&$.76$\\%
$1\,600$&$$&$1$&$.22\cdot 10^{-4}$&$1$&$.56$&$3$&$.14\cdot 10^{-4}$&$1$&$.56$&$3$&$.09\cdot 10^{-4}$&$1$&$.56$&$3$&$.28\cdot 10^{-4}$&$1$&$.56$\\%
$3\,200$&$$&$1$&$.17\cdot 10^{-4}$&$1$&$.04$&$3$&$.03\cdot 10^{-4}$&$1$&$.04$&$2$&$.98\cdot 10^{-4}$&$1$&$.04$&$3$&$.17\cdot 10^{-4}$&$1$&$.04$\\\bottomrule %
\end {tabular}%

    \begin {tabular}{r<{\pgfplotstableresetcolortbloverhangright }@{}l<{\pgfplotstableresetcolortbloverhangleft }r<{\pgfplotstableresetcolortbloverhangright }@{}l<{\pgfplotstableresetcolortbloverhangleft }r<{\pgfplotstableresetcolortbloverhangright }@{}l<{\pgfplotstableresetcolortbloverhangleft }r<{\pgfplotstableresetcolortbloverhangright }@{}l<{\pgfplotstableresetcolortbloverhangleft }r<{\pgfplotstableresetcolortbloverhangright }@{}l<{\pgfplotstableresetcolortbloverhangleft }r<{\pgfplotstableresetcolortbloverhangright }@{}l<{\pgfplotstableresetcolortbloverhangleft }r<{\pgfplotstableresetcolortbloverhangright }@{}l<{\pgfplotstableresetcolortbloverhangleft }r<{\pgfplotstableresetcolortbloverhangright }@{}l<{\pgfplotstableresetcolortbloverhangleft }r<{\pgfplotstableresetcolortbloverhangright }@{}l<{\pgfplotstableresetcolortbloverhangleft }}%
\toprule \multicolumn {2}{c}{$J$}&\multicolumn {2}{c}{$E_{L^2}$}&\multicolumn {2}{c}{$R_{L^2}$}&\multicolumn {2}{c}{$E_C$}&\multicolumn {2}{c}{$R_{C}$}&\multicolumn {2}{c}{$E_{L^2,\,{\rm rel}}$}&\multicolumn {2}{c}{$R_{L^2,\,{\rm rel}}$}&\multicolumn {2}{c}{$E_{C,\,{\rm rel}}$}&\multicolumn {2}{c}{$R_{C,\,{\rm rel}}$}\\\midrule %
$200$&$$&$1$&$.98\cdot 10^{-2}$&--&&$7$&$.11\cdot 10^{-2}$&--&&$10$&$.98$&--&&$3$&$.48$&--&\\%
$400$&$$&$2$&$.86\cdot 10^{-3}$&$6$&$.91$&$8$&$.87\cdot 10^{-3}$&$8$&$.02$&$2$&$.75$&$4$&$$&$0$&$.84$&$4$&$.16$\\%
$800$&$$&$7$&$.09\cdot 10^{-4}$&$4$&$.04$&$1$&$.88\cdot 10^{-3}$&$4$&$.71$&$0$&$.69$&$3$&$.96$&$0$&$.21$&$4$&$.04$\\%
$1\,600$&$$&$1$&$.77\cdot 10^{-4}$&$4$&$.01$&$6$&$.53\cdot 10^{-4}$&$2$&$.88$&$0$&$.17$&$3$&$.97$&$5$&$.16\cdot 10^{-2}$&$4$&$.01$\\%
$3\,200$&$$&$1$&$.18\cdot 10^{-4}$&$1$&$.5$&$3$&$.80\cdot 10^{-4}$&$1$&$.72$&$4$&$.39\cdot 10^{-2}$&$3$&$.98$&$1$&$.29\cdot 10^{-2}$&$4$&$$\\\bottomrule %
\end {tabular}%

    \begin {tabular}{r<{\pgfplotstableresetcolortbloverhangright }@{}l<{\pgfplotstableresetcolortbloverhangleft }r<{\pgfplotstableresetcolortbloverhangright }@{}l<{\pgfplotstableresetcolortbloverhangleft }r<{\pgfplotstableresetcolortbloverhangright }@{}l<{\pgfplotstableresetcolortbloverhangleft }r<{\pgfplotstableresetcolortbloverhangright }@{}l<{\pgfplotstableresetcolortbloverhangleft }r<{\pgfplotstableresetcolortbloverhangright }@{}l<{\pgfplotstableresetcolortbloverhangleft }r<{\pgfplotstableresetcolortbloverhangright }@{}l<{\pgfplotstableresetcolortbloverhangleft }r<{\pgfplotstableresetcolortbloverhangright }@{}l<{\pgfplotstableresetcolortbloverhangleft }r<{\pgfplotstableresetcolortbloverhangright }@{}l<{\pgfplotstableresetcolortbloverhangleft }r<{\pgfplotstableresetcolortbloverhangright }@{}l<{\pgfplotstableresetcolortbloverhangleft }}%
\toprule \multicolumn {2}{c}{$J$}&\multicolumn {2}{c}{$E_{L^2}$}&\multicolumn {2}{c}{$R_{L^2}$}&\multicolumn {2}{c}{$E_C$}&\multicolumn {2}{c}{$R_{C}$}&\multicolumn {2}{c}{$E_{L^2,\,{\rm rel}}$}&\multicolumn {2}{c}{$R_{L^2,\,{\rm rel}}$}&\multicolumn {2}{c}{$E_{C,\,{\rm rel}}$}&\multicolumn {2}{c}{$R_{C,\,{\rm rel}}$}\\\midrule %
$200$&$$&$1$&$.93\cdot 10^{-2}$&--&&$5$&$.65\cdot 10^{-2}$&--&&$3$&$.99$&--&&$1$&$.26$&--&\\%
$400$&$$&$1$&$.30\cdot 10^{-3}$&$14$&$.8$&$4$&$.30\cdot 10^{-3}$&$13$&$.15$&$0$&$.52$&$7$&$.74$&$0$&$.16$&$8$&$.06$\\%
$800$&$$&$1$&$.92\cdot 10^{-4}$&$6$&$.79$&$6$&$.12\cdot 10^{-4}$&$7$&$.02$&$6$&$.57\cdot 10^{-2}$&$7$&$.85$&$1$&$.96\cdot 10^{-2}$&$8$&$.01$\\%
$1\,600$&$$&$1$&$.22\cdot 10^{-4}$&$1$&$.57$&$3$&$.24\cdot 10^{-4}$&$1$&$.89$&$8$&$.30\cdot 10^{-3}$&$7$&$.92$&$2$&$.45\cdot 10^{-3}$&$8$&$$\\%
$3\,200$&$$&$1$&$.17\cdot 10^{-4}$&$1$&$.04$&$3$&$.01\cdot 10^{-4}$&$1$&$.08$&$1$&$.07\cdot 10^{-3}$&$7$&$.75$&$3$&$.15\cdot 10^{-4}$&$7$&$.79$\\\bottomrule %
\end {tabular}%

}
\caption{\small Errors and their ratios for the numerical solutions using the DTBC (upper), the SDTBC (middle) and the ISDTBC (lower)
in dependence with $J$, for $\theta=\frac{1}{12}$ and $M=6000$}
\label{tab:EX01:DTBC:J}
\end{table}

\begin{table}
\centering{
    \begin {tabular}{r<{\pgfplotstableresetcolortbloverhangright }@{}l<{\pgfplotstableresetcolortbloverhangleft }r<{\pgfplotstableresetcolortbloverhangright }@{}l<{\pgfplotstableresetcolortbloverhangleft }r<{\pgfplotstableresetcolortbloverhangright }@{}l<{\pgfplotstableresetcolortbloverhangleft }r<{\pgfplotstableresetcolortbloverhangright }@{}l<{\pgfplotstableresetcolortbloverhangleft }r<{\pgfplotstableresetcolortbloverhangright }@{}l<{\pgfplotstableresetcolortbloverhangleft }r<{\pgfplotstableresetcolortbloverhangright }@{}l<{\pgfplotstableresetcolortbloverhangleft }r<{\pgfplotstableresetcolortbloverhangright }@{}l<{\pgfplotstableresetcolortbloverhangleft }r<{\pgfplotstableresetcolortbloverhangright }@{}l<{\pgfplotstableresetcolortbloverhangleft }r<{\pgfplotstableresetcolortbloverhangright }@{}l<{\pgfplotstableresetcolortbloverhangleft }}%
\toprule \multicolumn {2}{c}{$M$}&\multicolumn {2}{c}{$E_{L^2}$}&\multicolumn {2}{c}{$R_{L^2}$}&\multicolumn {2}{c}{$E_{C}$}&\multicolumn {2}{c}{$R_{C}$}&\multicolumn {2}{c}{$E_{L^2,\,{\rm rel}}$}&\multicolumn {2}{c}{$R_{L^2,\,{\rm rel}}$}&\multicolumn {2}{c}{$E_{C,\,{\rm rel}}$}&\multicolumn {2}{c}{$R_{C,\,{\rm rel}}$}\\\midrule %
$375$&$$&$3$&$.00\cdot 10^{-2}$&--&&$7$&$.73\cdot 10^{-2}$&--&&$7$&$.77\cdot 10^{-2}$&--&&$8$&$.18\cdot 10^{-2}$&--&\\%
$750$&$$&$7$&$.50\cdot 10^{-3}$&$4$&$$&$1$&$.93\cdot 10^{-2}$&$4$&$$&$1$&$.91\cdot 10^{-2}$&$4$&$.07$&$2$&$.03\cdot 10^{-2}$&$4$&$.03$\\%
$1\,500$&$$&$1$&$.88\cdot 10^{-3}$&$4$&$$&$4$&$.83\cdot 10^{-3}$&$4$&$$&$4$&$.76\cdot 10^{-3}$&$4$&$.02$&$5$&$.06\cdot 10^{-3}$&$4$&$.01$\\%
$3\,000$&$$&$4$&$.69\cdot 10^{-4}$&$4$&$$&$1$&$.21\cdot 10^{-3}$&$4$&$$&$1$&$.19\cdot 10^{-3}$&$4$&$$&$1$&$.26\cdot 10^{-3}$&$4$&$$\\%
$6\,000$&$$&$1$&$.17\cdot 10^{-4}$&$3$&$.99$&$3$&$.03\cdot 10^{-4}$&$3$&$.99$&$2$&$.98\cdot 10^{-4}$&$3$&$.99$&$3$&$.17\cdot 10^{-4}$&$3$&$.99$\\\bottomrule %
\end {tabular}%

    \begin {tabular}{r<{\pgfplotstableresetcolortbloverhangright }@{}l<{\pgfplotstableresetcolortbloverhangleft }r<{\pgfplotstableresetcolortbloverhangright }@{}l<{\pgfplotstableresetcolortbloverhangleft }r<{\pgfplotstableresetcolortbloverhangright }@{}l<{\pgfplotstableresetcolortbloverhangleft }r<{\pgfplotstableresetcolortbloverhangright }@{}l<{\pgfplotstableresetcolortbloverhangleft }r<{\pgfplotstableresetcolortbloverhangright }@{}l<{\pgfplotstableresetcolortbloverhangleft }r<{\pgfplotstableresetcolortbloverhangright }@{}l<{\pgfplotstableresetcolortbloverhangleft }r<{\pgfplotstableresetcolortbloverhangright }@{}l<{\pgfplotstableresetcolortbloverhangleft }r<{\pgfplotstableresetcolortbloverhangright }@{}l<{\pgfplotstableresetcolortbloverhangleft }r<{\pgfplotstableresetcolortbloverhangright }@{}l<{\pgfplotstableresetcolortbloverhangleft }}%
\toprule \multicolumn {2}{c}{$M$}&\multicolumn {2}{c}{$E_{L^2}$}&\multicolumn {2}{c}{$R_{L^2}$}&\multicolumn {2}{c}{$E_{C}$}&\multicolumn {2}{c}{$R_{C}$}&\multicolumn {2}{c}{$E_{L^2,\,{\rm rel}}$}&\multicolumn {2}{c}{$R_{L^2,\,{\rm rel}}$}&\multicolumn {2}{c}{$E_{C,\,{\rm rel}}$}&\multicolumn {2}{c}{$R_{C,\,{\rm rel}}$}\\\midrule %
$375$&$$&$3$&$.00\cdot 10^{-2}$&--&&$7$&$.73\cdot 10^{-2}$&--&&$7$&$.91\cdot 10^{-2}$&--&&$8$&$.17\cdot 10^{-2}$&--&\\%
$750$&$$&$7$&$.50\cdot 10^{-3}$&$4$&$$&$1$&$.93\cdot 10^{-2}$&$4$&$$&$4$&$.67\cdot 10^{-2}$&$1$&$.69$&$2$&$.02\cdot 10^{-2}$&$4$&$.04$\\%
$1\,500$&$$&$1$&$.87\cdot 10^{-3}$&$4$&$$&$4$&$.83\cdot 10^{-3}$&$4$&$$&$4$&$.40\cdot 10^{-2}$&$1$&$.06$&$1$&$.29\cdot 10^{-2}$&$1$&$.57$\\%
$3\,000$&$$&$4$&$.69\cdot 10^{-4}$&$4$&$$&$1$&$.26\cdot 10^{-3}$&$3$&$.84$&$4$&$.39\cdot 10^{-2}$&$1$&$$&$1$&$.29\cdot 10^{-2}$&$1$&$$\\%
$6\,000$&$$&$1$&$.18\cdot 10^{-4}$&$3$&$.97$&$3$&$.80\cdot 10^{-4}$&$3$&$.31$&$4$&$.39\cdot 10^{-2}$&$1$&$$&$1$&$.29\cdot 10^{-2}$&$1$&$$\\\bottomrule %
\end {tabular}%

    \begin {tabular}{r<{\pgfplotstableresetcolortbloverhangright }@{}l<{\pgfplotstableresetcolortbloverhangleft }r<{\pgfplotstableresetcolortbloverhangright }@{}l<{\pgfplotstableresetcolortbloverhangleft }r<{\pgfplotstableresetcolortbloverhangright }@{}l<{\pgfplotstableresetcolortbloverhangleft }r<{\pgfplotstableresetcolortbloverhangright }@{}l<{\pgfplotstableresetcolortbloverhangleft }r<{\pgfplotstableresetcolortbloverhangright }@{}l<{\pgfplotstableresetcolortbloverhangleft }r<{\pgfplotstableresetcolortbloverhangright }@{}l<{\pgfplotstableresetcolortbloverhangleft }r<{\pgfplotstableresetcolortbloverhangright }@{}l<{\pgfplotstableresetcolortbloverhangleft }r<{\pgfplotstableresetcolortbloverhangright }@{}l<{\pgfplotstableresetcolortbloverhangleft }r<{\pgfplotstableresetcolortbloverhangright }@{}l<{\pgfplotstableresetcolortbloverhangleft }}%
\toprule \multicolumn {2}{c}{$M$}&\multicolumn {2}{c}{$E_{L^2}$}&\multicolumn {2}{c}{$R_{L^2}$}&\multicolumn {2}{c}{$E_{C}$}&\multicolumn {2}{c}{$R_{C}$}&\multicolumn {2}{c}{$E_{L^2,\,{\rm rel}}$}&\multicolumn {2}{c}{$R_{L^2,\,{\rm rel}}$}&\multicolumn {2}{c}{$E_{C,\,{\rm rel}}$}&\multicolumn {2}{c}{$R_{C,\,{\rm rel}}$}\\\midrule %
$375$&$$&$3$&$.00\cdot 10^{-2}$&--&&$7$&$.73\cdot 10^{-2}$&--&&$7$&$.77\cdot 10^{-2}$&--&&$8$&$.18\cdot 10^{-2}$&--&\\%
$750$&$$&$7$&$.50\cdot 10^{-3}$&$4$&$$&$1$&$.93\cdot 10^{-2}$&$4$&$$&$1$&$.91\cdot 10^{-2}$&$4$&$.07$&$2$&$.03\cdot 10^{-2}$&$4$&$.03$\\%
$1\,500$&$$&$1$&$.88\cdot 10^{-3}$&$4$&$$&$4$&$.83\cdot 10^{-3}$&$4$&$$&$4$&$.76\cdot 10^{-3}$&$4$&$.02$&$5$&$.06\cdot 10^{-3}$&$4$&$.01$\\%
$3\,000$&$$&$4$&$.69\cdot 10^{-4}$&$4$&$$&$1$&$.21\cdot 10^{-3}$&$4$&$$&$1$&$.44\cdot 10^{-3}$&$3$&$.31$&$1$&$.26\cdot 10^{-3}$&$4$&$$\\%
$6\,000$&$$&$1$&$.17\cdot 10^{-4}$&$3$&$.99$&$3$&$.01\cdot 10^{-4}$&$4$&$.01$&$1$&$.07\cdot 10^{-3}$&$1$&$.34$&$3$&$.15\cdot 10^{-4}$&$4$&$.01$\\\bottomrule %
\end {tabular}%

}
\caption{\small Errors and and their ratios for the numerical solutions using the DTBC (upper), the SDTBC (middle) and the ISDTBC (lower)
in dependence with $M$, for $\theta=\frac{1}{12}$ and $J=3200$}
\label{tab:EX01:DTBC:M}
\end{table}

\end{document}